\documentclass{article}
\usepackage[left=0.76in,right=0.75in,top=0.76in,bottom=0.78in]{geometry}

\usepackage{cite}
\usepackage{amsmath,amssymb,amsfonts}
\usepackage{algorithmic}
\usepackage{graphicx}
\usepackage{textcomp}
\usepackage{xcolor}
\usepackage{booktabs}
 \usepackage{psfrag}
\usepackage{amsthm}
 \usepackage{caption}
\usepackage{subcaption}
\usepackage{balance}
\usepackage{multirow}

\usepackage[linesnumbered,ruled,vlined,algo2e]{algorithm2e}

\SetKwInput{KwInput}{Input}                
\SetKwInput{KwOutput}{Output}             

\usepackage{dsfont}
\usepackage{tikz}
\usetikzlibrary{shapes.geometric}
\usepackage{pgfplots}
\pgfplotsset{width=0.45\textwidth}

\newtheorem{thm}{Theorem}[section]

\newtheorem{defn}{Definition}[section]

\newtheorem{rem}{Remark}
\newtheorem{assum}{Assumption}[section]

\usepackage{graphicx}

\newcommand{\T}{^\mathsf{T}} 
\renewcommand{\d}{\mathrm{d}} 
\newcommand{\R}{\mathds{R}} 

\DeclareMathOperator*{\argmax}{arg\,max}

\def\BibTeX{{\rm B\kern-.05em{\sc i\kern-.025em b}\kern-.08em
    T\kern-.1667em\lower.7ex\hbox{E}\kern-.125emX}}
\begin{document}

\title{\LARGE \bf Automatic scenario generation for efficient solution of robust optimal control problems\\
\thanks{This work has received funding from the EPSRC (Engineering and Physical Sciences) under the Active Building Centre
 project (reference number: EP/V012053/1). M. Zagorowska also acknowledges funding from the European Research Council (ERC) under the H2020 Advanced Grant no. 787845 (OCAL).}
}

\author{Marta Zagorowska, Paola Falugi, Edward O'Dwyer, and Eric C. Kerrigan
\thanks{M. A. Zagorowska was with the Department of Electrical and Electronic Engineering, Imperial College London, currently with Automatic Control Laboratory, ETH Zurich 
        {\tt\small m.zagorowska@imperial.ac.uk, mzagorowska@ethz.ch}}%
\thanks{P. Falugi was with the Department of Electrical and Electronic Engineering, Imperial College London, currently with the University of East London
        {\tt\small p.falugi@imperial.ac.uk,
        p.falugi@uel.ac.uk}}%
\thanks{E. O'Dwyer is with the Department of Chemical Engineering, Imperial College London 
        {\tt\small e.odwyer@imperial.ac.uk}}%
\thanks{E. C. Kerrigan is with the Department of Electrical and Electronic Engineering and with Department of Aeronautics, Imperial College London
        {\tt\small e.kerrigan@imperial.ac.uk}}%
}

\maketitle

\begin{abstract}
Existing methods for nonlinear robust control often use scenario-based approaches to formulate the control problem as large nonlinear optimization problems. The optimization problems are challenging to solve due to their size, especially if the control problems include time-varying uncertainty. This paper draws from local reduction methods used in semi-infinite optimization to solve robust optimal control problems with parametric and time-varying uncertainty. By iteratively adding interim  worst-case scenarios to the problem, methods based on local reduction provide a way to manage the total number of scenarios. We show that the local reduction method for optimal control problems consists of solving a series of simplified optimal control problems to find worst-case constraint violations. In particular, we present examples where local reduction methods find worst-case scenarios that are not on the boundary of the uncertainty set. We also provide bounds on the error if local solvers are used. The proposed approach is illustrated with two case studies with parametric and additive time-varying uncertainty. In the first case study, the number of scenarios obtained from local reduction is 101, smaller than in the case when all $2^{14+3\times192}$ extreme scenarios are considered. In the second case study, the number of scenarios obtained from the local reduction is two compared to 512 extreme scenarios. Our approach was able to satisfy the constraints both for parametric uncertainty and time-varying disturbances, whereas approaches from literature either violated the constraints or became computationally expensive. 
\end{abstract}

\section{Introduction}

Robust optimal control problems are often solved using a scenario-based approach, where each \emph{scenario} corresponds to a separate realization of uncertainty. Increasing the number of scenarios improves robustness, while increasing the size of the optimization problems. Mitigating the size of the problem by reducing the number of scenarios requires knowledge about how the uncertainty affects the system which is a challenge especially if uncertainty varies with time. This paper draws from approaches used in semi-infinite optimisation to solve robust optimal control problems by selecting scenarios in an efficient way.

\subsection{Background}
To ensure that the optimisation problems resulting from scenario-based approaches to robust control are tractable, the number of scenarios must be limited \cite{scenario_Calafiore2006}. Usually, the choice of scenarios is done from experience \cite{scenario_Grammatico2015,Fundamentals_Haaberg2019} and requires knowledge about both the controlled system and the uncertainty to ensure that the chosen scenarios guarantee robustness. A recent review of scenario-based methods
indicated that scenario selection is highly affected by the knowledge about the uncertainty distribution\cite{scenario_Campi2021}. In practice, to tackle problems with limited knowledge about the uncertainty, it is often assumed that the worst-case scenarios lie on the boundary of the uncertainty set \cite{Cutting_Mutapcic2009,Handling_Lucia2014,Monotonicity_Vuffray2015,Sensitivity_Thombre2021}. The worst-case scenario in nonlinear systems may lie in the interior of the uncertainty range\cite{Interval_Krasnochtanova2010,Robust_Puschke2018}. In this paper we show that the worst-case scenarios may be inside the uncertainty set, even for a linear dynamic system and present a method for choosing potential worst-case scenarios assuming limited knowledge about uncertainty. 

Systematic approaches to choosing scenarios for time-varying uncertainty are usually based on creating \emph{scenario trees} \cite{Scenario_Heitsch2009,Sensitivity_Thombre2021}. In these approaches, a large set of scenarios is chosen at the beginning of the time horizon. The number of elements in the set of scenarios increases combinatorially~\cite{Min_Scokaert1998}. However, there is no guarantee that the chosen set of scenarios includes the actual worst-case scenarios\cite{Robust_Puschke2018}. To overcome this drawback, we propose to use a method derived from semi-infinite optimisation to iteratively add new scenarios to the current set, to provide more flexibility in adjusting the set of scenarios and finding worst-case scenarios. 

Previous studies provided an in-depth review of semi-infinite optimization methods \cite{Semi_Hettich1993,Semi_Hettich2009,review_Hettich1983,adaptive_Seidel2020,Recent_Djelassi2021}. In particular, it has been indicated\cite{Semi_Hettich1993} that \emph{local reduction} methods allow overcoming the dependence on the initial choice of scenarios. In these methods, the set of scenarios is created iteratively by alternating between solving an optimisation problem with the current set of scenarios and solving interim optimisation problems to find the maximal violation of constraints and extend the set of scenarios \cite{Infinitely_Blankenship1976}. Therefore local reduction methods enable adding scenarios that may not have been considered at the beginning, such as scenarios from the interior of the uncertainty set. 

The iterative approach based on alternating between optimisation problems has been already used for robust control, which suggests potential usefulness of local reduction. The D-K iteration found in $\mu$-synthesis consists in alternating between the synthesis of an $\mathcal{H}_{\infty}$ controller and minimisation of the singular value for the corresponding controller \cite{Multivariable_Skogestad2007}. However, the D-K iteration method depends on the optimality of each step, which indicates that the local reduction method will also be affected by how the interim optimisation problems are solved. In particular, due to the need to solve to global optimality the interim problems\cite{Recent_Djelassi2021}, the complexity of local reduction methods prevented their use in robust optimal control problems. An approach based on simulated annealing has been proposed to facilitate finding a global solution\cite{reduction_Pereira2009}. We show that the local reduction method provides good results even if local solvers are used and provide bounds on the solution. 

Existing applications of semi-infinite optimisation methods in control systems with uncertainty are limited. Semi-infinite optimization methods have been used for optimal control\cite{Semi_Hauser2018} to find optimal trajectories for robotic arms. However, the authors considered only exogenous uncertainty due to obstacles that did not affect the dynamics of the controlled systems. Parametric uncertainty for linear systems was considered\cite{Robust_Katselis2012,parallel_Zakovic2003} to apply semi-infinite optimisation methods for model identification. However, these works did not consider time-varying uncertainty. Similarly, semi-infinite optimisation methods has been used to solve an optimal control problem with only parametric uncertainty\cite{Robust_Puschke2018}. Time-varying uncertainty was considered in works where local reduction was applied to find the interim worst-case scenarios\cite{Sensitivity_Thombre2021}. However, the authors assumed that at every time step the number of possible scenarios was finite.

\subsection{Contributions}

The main contributions of the current work are:
\begin{itemize}
    \item Formulation of a local reduction method as a way of automatically generating scenarios in robust nonlinear control problems if time-varying and constant uncertainties are present, 
    \item Formulation of error bounds on constraint violation if local optimization solvers are used in local reduction,
    \item Numerical demonstration of local reduction methods for solving nonlinear robust optimal control problems with parametric uncertainty and linear robust optimal control problems with both parametric and time-varying uncertainty.
\end{itemize}
A preliminary version of the application of local reduction for scenario generation is available\cite{Efficient_Zagorowska2022a}.

The rest of the paper is structured  as follows. Section~\ref{sec:ProblemFormulation} introduces robust optimal control problems. Section~\ref{sec:LocalReduction} presents the new method for solving robust optimal control problems. The numerical results are shown in Section~\ref{sec:NumericalResults} where we compared the results obtained with the scenarios from local reduction to three commonly used approaches: nominal case with no uncertainty, random case with randomly drawn realisations of uncertainty, boundary case with only extreme values. The paper ends with conclusions in Section~\ref{sec:Conclusions}.

\section{Problem formulation}
\label{sec:ProblemFormulation}

\subsection{Semi-infinite optimization problem}
\label{sec:semiinfProblem}
A semi-infinite optimization problem is formulated as:
\begin{subequations} \label{eqn:SemiInf}
\begin{align}
\mathcal{Q}:\quad &\min_{\theta\in\mathcal{A}} \quad Q(\theta)
    \label{eqn:SemiinfCost}\\
&\text{subject to } R(\theta,\rho)\leq 0 \label{eq:Semiinfmapping} \text{ for all }\rho\in\mathcal{B}
\end{align}
\end{subequations}
where $\mathcal{A}\subset \R^{n_{\theta}}$ and $\mathcal{B}\subset\R^{n_{\rho}}$ are nonempty and compact sets, and $Q$ and $R$ are continuous functions of their respective arguments \cite{Infinitely_Blankenship1976}. The problem \eqref{eqn:SemiInf} has a finite number of variables $\theta$ but includes an infinite number of constraints if $\mathcal{B}$ has an infinite number of points. In particular, $\mathcal{B}$ may be uncountable. 

One approach to remove the infinite number of constraints consists in rewriting the constraint \eqref{eq:Semiinfmapping} as:
\begin{equation}
    S(\theta):=\max_{\rho\in\mathcal{B}} R(\theta,\rho)\leq 0. \label{eq:SemiinfmappingMax}
\end{equation}
The challenge in solving the equivalent problem with constraint \eqref{eq:SemiinfmappingMax} is in non-differentiability of the function $S(\cdot)$. The local reduction method proposed by \cite{Infinitely_Blankenship1976} allows overcoming the non-differentiability of $S(\cdot)$ by sequentially solving \eqref{eqn:SemiInf} with finite subsets of constraints taken from $\mathcal{B}$. 

The main challenge in formulating robust optimal control problems as semi-infinite optimisation lies in inclusion of system dynamics in the form of equality constraints. In this paper, we show that optimal control problems can be formulated as semi-infinite optimization problems and solved using the local reduction method\cite{Infinitely_Blankenship1976}.  

\subsection{Dynamic system with uncertainty}
The system to be controlled is described by a nonlinear difference equation with time-varying uncertainty $w_k\in\mathbb{W}\subset \R^{n_w}$ and constant uncertainty $d\in\mathbb{D}\subset\R^{n_d}$:
\begin{equation}
x_{k+1}=f_k(x_k,u_k,w_k,d)
\label{eq:InitialDynamics}
\end{equation}
where $f_k$ is continuously differentiable.
The state $x_0$ at time zero is w.l.o.g.\ assumed to be equal to a given~$\hat{x}$. 

The control trajectory $\mathbf{u}:=(u_0,\ldots,u_{N-1})$ is generated by a causal dynamic feedback policy 
\[
u_k:=\pi_k(x_0,\ldots,x_k;q_0,\ldots,q_k,r)
\] 
that is parameterised by $\mathbf{q}:=(q_0,q_1,\ldots,q_{N-1})\in\mathbb{R}^{n_q}$ and $r\in\R^{n_r}$. The state trajectory $\mathbf{x}:=(x_0,\ldots,x_N)$. The time-varying uncertainty $w_k$ at time~$k$ and the constant uncertainty $d$ affect the dynamics in both an additive and non-additive way, and take on values from compact and uncountable (infinite cardinality) sets. Uncertainty in the measured value of $x_k$ can be modelled by a suitably-defined choice of $f_k$, $\pi_k$ and $w_k$. In this work, we assume that the structure of the dynamic feedback policy, and hence the parameterisation, is already defined.

A trajectory $(\mathbf{x},\mathbf{u})$ satisfying the dynamics \eqref{eq:InitialDynamics} and control policy for a given  parameterization $(\mathbf{q},r)$ and realisation of uncertainty $(\mathbf{w}$, $d$), where the trajectory $\mathbf{w}:=(w_0,\ldots,w_{N-1}) \in \mathbb{W}^N:=\mathbb{W}\times\cdots\times\mathbb{W}$,  is defined as:
\begin{equation}
\label{eq:Trajectory}
 \mathbf{z}(\mathbf{q},r,\mathbf{w},d):=  \Big\lbrace (\mathbf{x},\mathbf{u}) \mid x_0=\hat{x},
x_{k+1}=f_k(x_k,u_k,w_k,d),
u_k=\pi_k(x_0,\ldots,x_k;q_0,\ldots,q_k,r),
k=0,1,\ldots,N-1 \Big\rbrace. 
\end{equation}

\subsection{Robust optimal control problem}
\label{sec:RobustOptimalControl}

\subsubsection{Objective function and constraints}
\label{sec:ObjCstr}

The cost function for the optimal control problem over a horizon of length $N$ is:
\begin{equation}
J_N(\mathbf{x},\mathbf{u},\mathbf{w}, d):=J_f(x_N,w_N,d)+\sum\limits_{k=0}^{N-1}\ell_k(x_k,u_k,w_k,d).
\label{eq:Objective}
\end{equation}
Both the terminal cost function $J_f(\cdot,\cdot,\cdot)$ and stage cost  $\ell_k(\cdot,\cdot,\cdot,\cdot)$ are continuously differentiable and depend on the uncertainty $\textbf{w}$ and $d$. The objective of the optimal control problem is to find a feedback policy $\pi$ for system~\eqref{eq:InitialDynamics} such that the worst-case cost in \eqref{eq:Objective} is minimized and the constraints
\begin{equation}
g_k(x_k,u_k,w_k,d)\leq 0
\label{eq:InitCstr}
\end{equation}
are satisfied for all time instants $k=0,\ldots,N-1$, all states~$\mathbf{x}$, control~$\mathbf{u}$, uncertainty~$\mathbf{w}$~and~$d$. The vector function of $n_{g}$ components, $g_k(\cdot,\cdot,\cdot,\cdot)$, is continuously differentiable and depends on uncertainty $\mathbf{w}$ and $d$. Note that a constraint on~ $x_N$ can be included by incorporating $f_{N-1}$ in a suitable definition of $g_{N-1}$.

To ensure that the optimal control problem with the objective \eqref{eq:Objective} and the constraints \eqref{eq:InitCstr} is well-defined over the horizon $N$, we introduce the following assumption on the trajectories \eqref{eq:Trajectory}:
\begin{assum}
\label{ass:Bounded}
The trajectories \eqref{eq:Trajectory} are bounded over a finite horizon $N$:
\begin{equation}
\forall N\in\R, \mathbf{q}\in\R^{n_q},r\in\R^{n_r}, \mathbf{w}\in\mathbb{W}^N,d\in\mathbb{D}\; \exists \varsigma_z\in\R:\; \|\mathbf{z}(\mathbf{q},r,\mathbf{w},d)\|\leq \varsigma_z
\end{equation}
\end{assum}
\noindent  We note that the boundedness of \eqref{eq:Trajectory} need not imply stability of the dynamics \eqref{eq:InitialDynamics}.

\subsubsection{Semi-infinite formulation}
\label{sec:SemiInfForm}

Given a set of uncertainties 
$
\mathbb{H} \subseteq \mathbb{W}^N\times\mathbb{D},
$
the problem in this work is stated as:
\begin{subequations}
\label{eq:Plifted}
\begin{align}
 \mathcal{P}_N(\mathbb{H}): \min_{{\substack{\mathbf{q},r\\ \mathbf{x}^i,\mathbf{u}^i,\\i\in\mathbb{J}}}}& \max_{\substack{\mathbf{w}^i, d^i\\i\in\mathbb{J}}}\  J_N(\mathbf{x}^i,\mathbf{u}^i,\mathbf{w}^i, d^i) \label{eq:CostLifted}
   \\
 \text{s.t. } 
g_k(x_k^i,u_k^i,w_k^i,d^i)\leq 0,\ &\forall i\in\mathbb{J}, k=0,\ldots,N-1\\
(\mathbf{x}^i,\mathbf{u}^i) = \mathbf{z}(\mathbf{q},r,\mathbf{w}^i,d^i),\ &\forall i\in\mathbb{J} \label{eq:TrajectoryAll}
\end{align}
\end{subequations}
where $\mathbb{J} := \lbrace 1,\ldots,\operatorname{card} \mathbb{H}\rbrace$ and $(\mathbf{x}^i,\mathbf{u}^i)$ is the state and input trajectory associated with the $i^\text{th}$ disturbance realisation $(\mathbf{w}^i,d^i)$ such that $$\mathbb{H} = \bigcup_{i\in\mathbb{J}}\{(\mathbf{w}^i,d^i)\}.$$

If $\mathbf{z}(\cdot)$ in \eqref{eq:TrajectoryAll} is linear jointly in all arguments, the problem \eqref{eq:Plifted} can often be solved using scenario-based methods for robust control\cite{scenario_Calafiore2006,Min_Scokaert1998}, provided additional convexity assumptions are satisfied by the uncertainty set $\mathbb{W}$. In this work, the dynamics from \eqref{eq:InitialDynamics} are nonlinear and $\mathbb{W}$ is only non-empty and compact. Moreover, the set $\mathbb{H}$ is assumed uncountable, which is a common case if the disturbances belong to a polytope.

\begin{thm}
\label{thm:MainResult}
The robust optimal control problem \eqref{eq:Plifted} is equivalent to the semi-infinite optimization problem \eqref{eqn:SemiInf} with $\theta:=(\mathbf{q},r,\gamma)$, where $\gamma$ is an additional scalar parameter characterizing the cost upper-bound, $\rho:=(\mathbf{w},d)$ and the sets $\mathcal{A}:=\R^{n_q}\times\R^{n_r}\times\R$, $\mathcal{B}:=\mathbb{H}$.
\end{thm}
\begin{proof}
In contrast to \eqref{eqn:SemiinfCost}, the objective function in \eqref{eq:CostLifted} contains uncertainty. Introducing $\gamma\in\R$, we rewrite \eqref{eq:Plifted} as:
\begin{subequations}
\label{eq:Pgamma}
\begin{align}
 \mathcal{P}_N(\mathbb{H}):\quad \min_{{\substack{\gamma,\mathbf{q},r\\ \mathbf{x}^i,\mathbf{u}^i, i\in\mathbb{J}}}}& \quad \gamma 
   \\
 \text{s.t. } 
g_k(x_k^i,u_k^i,w_k^i,d^i)\leq 0,\ &\forall i\in\mathbb{J},k=0,\ldots,N-1 \label{eq:PaugIneq}\\
(\mathbf{x}^i,\mathbf{u}^i) = \mathbf{z}(\mathbf{q},r,\mathbf{w}^i,d^i),\ &\forall i\in\mathbb{J}\\
J_N(\mathbf{x}^i,\mathbf{u}^i,\mathbf{w}^i, d^i)\leq \gamma,\ &\forall i\in\mathbb{J}\label{eq:CstrObjRef}
\end{align}
\end{subequations}
The problem \eqref{eq:Pgamma} has uncertainty exclusively in the constraints. If $\operatorname{card} \mathbb{H}$ is finite, then the problem \eqref{eq:Pgamma} is convenient to solve numerically using tailored efficient finite-dimensional optimisation methods that exploit the sparsity in the relevant Jacobians and Hessians. However, infinite cardinality of $ \mathbb{H}$ yields an infinite number of both constraints and variables, which means that the problem \eqref{eq:Pgamma} needs to be further reformulated to become \eqref{eqn:SemiInf}. Noticing that the constraint \eqref{eq:PaugIneq} is equivalent to
\begin{equation}
    \max_k g_k(x_k^i,u_k^i,w_k^i,d^i)\leq 0,\ \forall i\in\mathbb{J},
\end{equation}
we introduce
\begin{equation}
G(\mathbf{x}^i,\mathbf{u}^i,\mathbf{w}^i,d^i,\gamma):=\max\lbrace \max_{h,k}\; e_h^T g_k(x^i_k,u^i_k,w^i_k,d^i), J_N(\mathbf{x}^i,\mathbf{u}^i,\mathbf{w}^i,d^i)-\gamma \rbrace
\label{eq:AllConstraints}    
\end{equation}
In \eqref{eq:AllConstraints}, $e_h$ is the $h^\text{th}$ column of an identity matrix $\mathbb{I}_{n_g} $. Using~\eqref{eq:Trajectory} and \eqref{eq:AllConstraints}, we can write \eqref{eq:Pgamma} as:
\begin{subequations}
\label{eq:Prevised}
\begin{align}
&\mathcal{P}_N(\mathbb{H}):\quad  \min_{\mathbf{q},r,\gamma} \quad \gamma \label{eq:RevisedCost}
   \\
 &\quad \text{s.t. } 
G(\mathbf{z}(\mathbf{q},r,\mathbf{w},d),\mathbf{w},d,\gamma)\leq 0,\ \forall (\mathbf{w},d)\in\mathbb{H},
\label{eq:PreConstraints}
\end{align}
\end{subequations}
The problem \eqref{eq:Prevised} is equivalent to
\begin{subequations}
\label{eq:PrevisedMax}
\begin{align}
&\mathcal{P}_N(\mathbb{H}):\quad  \min_{\mathbf{q},r,\gamma} \quad \gamma
   \\
 &\quad \text{s.t. } 
\max_{(\mathbf{w},d)\in\mathbb{H}}G(\mathbf{z}(\mathbf{q},r,\mathbf{w},d),\mathbf{w},d,\gamma)\leq 0.
\label{eq:PreConstraintsMax}
\end{align}
\end{subequations}
Taking $\theta:=(\mathbf{q},r,\gamma)$ and $\rho:=(\mathbf{w},d)$ in \eqref{eq:Prevised} (similarly in \eqref{eq:PrevisedMax}) we obtain the form of \eqref{eqn:SemiInf} (similarly \eqref{eq:SemiinfmappingMax}).
\end{proof}

Theorem~\ref{thm:MainResult} makes no assumptions on the cardinality of the set $\mathbb{H}$, which is uncountable in general. As a result, problem \eqref{eq:Prevised} has an infinite number of constraints, in general, in a similar way to \eqref{eqn:SemiInf}. Using the fact that \eqref{eq:Prevised} is equivalent to \eqref{eqn:SemiInf}, and \eqref{eq:Plifted} is equivalent to~\eqref{eq:Prevised}, we note that solving \eqref{eq:Prevised} using methods developed for semi-infinite optimization of the form \eqref{eqn:SemiInf} is equivalent to solving the optimal control problem~\eqref{eq:Plifted}. Thus, Theorem~\ref{thm:MainResult} allows one to solve the optimal control problem~\eqref{eq:Plifted} as a semi-infinite optimization problem of the form~\eqref{eqn:SemiInf} using local reduction\cite{Infinitely_Blankenship1976}. 


\subsubsection{Robust solution}
We now introduce definitions of robust solutions that we are going to use in the remainder of the paper. First, we notice that the left-hand side of~\eqref{eq:PreConstraintsMax} is equivalent to:
\begin{equation}
    G_{\max}(\mathbf{q},r,\gamma,\mathbb{H}):=\max_{
{\substack{(\mathbf{w},d)\in\mathbb{H}\\ (\mathbf{x},\mathbf{u})=\mathbf{z}(\mathbf{q},r,\mathbf{w},d)}}
} G(\mathbf{x},\mathbf{u},\mathbf{w},d,\gamma)
    \label{eq:Gmax}
\end{equation}
Using \eqref{eq:Gmax} allows us to introduce the necessary definitions.

\begin{defn}[Scenario]
\label{def:Scenario}
A scenario is a realisation of the uncertainty $(\mathbf{w},d)\in\mathbb{W}^N\times\mathbb{D}$.
\end{defn}

\begin{defn}[Worst-case scenario]
\label{def:WorstCase}
A worst-case scenario for a given $(\mathbf{q},r,\gamma)$ is a realisation of the uncertainty $(\mathbf{w}^*,d^*)\in\mathbb{W}^N\times\mathbb{D}$ that maximises constraint violation:
\begin{equation}
(\mathbf{x}^*,\mathbf{u}^*,\mathbf{w}^*,d^*) \in \argmax_{
{\substack{(\mathbf{w},d)\in\mathbb{W}^N\times\mathbb{D}\\ (\mathbf{x},\mathbf{u})=\mathbf{z}(\mathbf{q},r,\mathbf{w},d)}}
} G(\mathbf{x},\mathbf{u},\mathbf{w},d,\gamma)
\end{equation}
\end{defn}

\begin{defn}[Robust solution]
A triple $(\mathbf{q}^*, r^*,\gamma^*)$ is called a \emph{robust solution} if
$G_{\max}(\mathbf{q}^*,r^*,\gamma^*,\mathbf{w},d)\leq 0$ for all $(\mathbf{w},d)\in \mathbb{W}^N\times\mathbb{D}$.

\end{defn}

\begin{defn}[Solution robust to $s$ scenarios]
A triple $(\mathbf{q}^*, r^*,\gamma^*)$ is called a \emph{robust solution to the $s$ scenarios in $\mathbb{H}$} if $G_{\max}(\mathbf{q}^*,r^*,\gamma^*,\mathbf{w},d)\leq 0$ for all $(\mathbf{w},d)\in \mathbb{H}$ and 
$\operatorname{card}\mathbb{H}=s$.
\label{def:Robusts}
\end{defn}

\section{Local reduction for optimal control}
\label{sec:LocalReduction}
We will now extend the local reduction methods\cite{Infinitely_Blankenship1976} to robust nonlinear optimal control.

\subsection{Algorithm}

The local reduction method\cite{Infinitely_Blankenship1976} consists in iteratively solving finite-dimensional optimization problems. We use the local reduction methods for the problem \eqref{eq:Pgamma} or \eqref{eq:Prevised} by iteratively solving optimal control problems parametrised by scenarios. The iterations alternate between solving minimization and maximization steps that will now be described.

\subsection{Minimization step}
The local reduction algorithm for robust optimal control is shown in Algorithm~\ref{alg:LocalReductionGeneral}. The proposed algorithm in iteration $j$ solves an optimal control problem of the form \eqref{eq:Pgamma} or \eqref{eq:Prevised} assuming that the number of scenarios $\operatorname{card} \mathbb{H}_j$ at step~$j$ is finite. The algorithm needs an initial guess for the parameters of the controller. For instance, the initial guess can be obtained by solving \eqref{eq:Prevised} for one scenario, i.e.\ $\operatorname{card}\mathbb{H}_1=1$. Alternatively, the initial guess can be found by solving \eqref{eq:Prevised} for a small number of scenarios, obtained for example from a coarse discretization of the uncertainty set \cite{transformation_Schwientek2020}. 

In the first step of Algorithm~\ref{alg:LocalReductionGeneral} (line 3), the algorithm checks whether worst-case scenarios exist that would lead to a violation of constraints~\eqref{eq:AllConstraints}. If no constraints are violated (line 4), the current parameters give a robust solution to the current set of scenarios $\mathbb{H}_j$. If there exists at least one violated constraint, then a scenario corresponding to the maximum constraint violation is added to the scenario set $\mathbb{H}_{j+1}$ in the next iteration (line 7). The new set $\mathbb{H}_{j+1}$  is then used to find a new set of control parameters (line~9). The algorithm ends if no new scenarios are added, i.e.\ $\operatorname{card}\mathbb{H}_{j}= \operatorname{card}\mathbb{H}_{j-1}$. 

In this work, any scenario corresponding to the maximum constraint violation can be added to the set of scenarios. However, it has been shown that computational performance may be improved if multiple scenarios are added\cite{global_Tsoukalas2008}.

\begin{algorithm2e}[t]
\SetAlgoLined
  \KwInput{Initial guess for $\mathbf{q}$, $r$, $\gamma$ and $\mathbb{H}_1\neq\emptyset$}
  \KwOutput{Optimal  $\mathbf{q}^*$, $r^*$, $\gamma^*$, set of scenarios $\mathbb{H}^*$}
  Set $\mathbf{q}^1\leftarrow\mathbf{q}$, $r^1\leftarrow r$, $\gamma^1\leftarrow\gamma$, $j\leftarrow 1$
  
  \Repeat{$\operatorname{card}\mathbb{H}_{j}= \operatorname{card}\mathbb{H}_{j-1}$}{
Compute $G_{\max}(\mathbf{q}^j,r^j,\gamma^j,\mathbb{W}^N\times\mathbb{D})$
and a maximizer $(\mathbf{x}^j,\mathbf{u}^j,\mathbf{w}^j,d^j)$ by solving~\eqref{eq:Gmax} with $\mathbb{H}=\mathbb{W}^N\times \mathbb{D}$.

\eIf{$G_{\max}(\mathbf{q}^j,r^j,\gamma^j,\mathbb{W}^N\times\mathbb{D})\leq 0$}{
\begin{equation*}
  \mathbb{H}_{j+1}\leftarrow\mathbb{H}_{j}
  \end{equation*}  }{
  Add new scenario 
  \begin{equation}
  \mathbb{H}_{j+1}\leftarrow\mathbb{H}_{j}\cup \{(\mathbf{w}^j,d^j)\}
  \label{eq:AddScenarios}
  \end{equation}
      
  Find a $(\mathbf{q}^{j+1},r^{j+1},\gamma^{j+1})$ that solves $\mathcal{P}_N(\mathbb{H}_{j+1})$ using~\eqref{eq:Pgamma} or \eqref{eq:Prevised}.
}
  Set $j\leftarrow j+1$

      Set $(\mathbf{q}^*,r^*,\gamma^*)\leftarrow(\mathbf{q}^j,r^j,\gamma^j)$  and $\mathbb{H}^*\leftarrow \mathbb{H}_j$.
 }
 \caption{Exact local reduction method\label{alg:LocalReductionGeneral}}
\end{algorithm2e}

\subsection{Maximization step}
The maximization step consists in solving \eqref{eq:Gmax} with $\mathbb{H}=\mathbb{W}^N\times\mathbb{D}$. Solving \eqref{eq:Gmax} is equivalent to solving $n_{g}\cdot (N-1)+1$ optimization problems, where $n_g$ denotes the number of elements in the vector function $g(\cdot)$ from constraints in \eqref{eq:InitCstr}. The algorithm is presented in Algorithm~\ref{alg:LocalReductionMaximisation}. Without loss of generality, we assume that the first constraint to include in the maximization problem corresponds to the reformulated objective function \eqref{eq:Objective}. A scenario that corresponds to maximal value of this constraint is added to an auxiliary set $\mathbb{K}$. The remaining $n_g\cdot (N-1)$ constraints are included as objectives in the respective maximization problems (line four to eight in Algorithm~\ref{alg:LocalReductionMaximisation}). Note that the problem corresponding to the objective (line 2) and all the problems corresponding to the constraints (line four to eight) can be solved in parallel.

All maximization problems are subject to the same equality constraints capturing the dynamics. This formulation allows us to treat the maximization problems as optimal control problems and preserve the sparsity of the relevant Jacobians and Hessians. We solve the maximization problems as optimal control problems where $\mathbf{q}$, $r$, and $\gamma$ are known parameters whereas $\mathbf{w}$ and $d$ are treated as unknown inputs. Thus, the maximization problems can be solved using any off-the-shelf solver for optimal control problems. 

\begin{algorithm2e}[t]
\SetAlgoLined
  \KwInput{Current values of $\mathbf{q}^j$, $r^j$, $\gamma^j$}
  \KwOutput{Worst case scenario $(\mathbf{w}^{j},d^{j})$ in iteration $j$}
  
        Find any $\mathbf{x}^*,\mathbf{u}^*,\mathbf{w}^*,d^*$   that solves:
        \begin{subequations}
        \begin{align}
        \max_{\mathbf{x},\mathbf{u},\mathbf{w},d} & \quad J_N(\mathbf{x},\mathbf{u},\mathbf{w},d)-\gamma^j\nonumber   \\
         \quad \text{s.t. } &
        (\mathbf{x},\mathbf{u}) = \mathbf{z}(\mathbf{q}^j,r^j,\mathbf{w},d)\nonumber\\
        & (\mathbf{w},d) \in \mathbb{W}^N\times\mathbb{D}\nonumber
        \end{align}
        \end{subequations}
        
        Set $\mathbb{K}\leftarrow{(\mathbf{w}^*,d^*,J_N(\mathbf{x}^*,\mathbf{u}^*,\mathbf{w}^*,d^*)-\gamma^j)}$

  \For{$h=1,\ldots,n_g$}{
    \For{$k=1,\ldots,N-1$}{
          Find any $\mathbf{x}^*,\mathbf{u}^*,\mathbf{w}^*,d^*$  that solves:
        \begin{subequations}
\begin{align}
\max_{\mathbf{x},\mathbf{u},\mathbf{w},d} & \quad e_h g_k(x_k,u_k,w_k,d)\nonumber   \\
 \quad \text{s.t. } &
(\mathbf{x},\mathbf{u}) = \mathbf{z}(\mathbf{q}^j,r^j,\mathbf{w},d)\nonumber\\
        & (\mathbf{w},d) \in \mathbb{W}^N\times\mathbb{D}\nonumber
\end{align}
\end{subequations}
        Set $\mathbb{K}\leftarrow\mathbb{K}\cup{(\mathbf{w}^*,d^*,e_h g_k(x_k^*,u_k^*,w_k^*,d^*))}$
        
        }
    }
    Set $v^*\leftarrow \max \lbrace v_3 \mid (v_1,v_2,v_3)\in\mathbb{K}\rbrace$;
    
    Choose any $(\mathbf{w}^{j},d^{j})\in \lbrace (v_1,v_2) \mid (v_1,v_2,v^*)\in \mathbb{K} \rbrace$
     \caption{Maximization - line 3 in Algorithm~\ref{alg:LocalReductionGeneral} \label{alg:LocalReductionMaximisation}}
\end{algorithm2e}
Solving~\eqref{eq:Gmax} with $\mathbb{H}=\mathbb{W}^N\times\mathbb{D}$ corresponds to lines four to eight in Algorithm~\ref{alg:LocalReductionMaximisation} and can be done by solving a number of finite-dimensional optimization problems in parallel\cite{parallel_Zakovic2003}. 

\subsection{Analysis}

\subsubsection{Convergence of Algorithm \ref{alg:LocalReductionGeneral}}
The convergence of local reduction method in the case of the form \eqref{eqn:SemiInf} was shown in several previous works \cite{Infinitely_Blankenship1976,Semi_Reemtsen1998,Global_Mitsos2011}. The authors required that the sets $\mathcal{A}$ and $\mathcal{B}$ in \eqref{eqn:SemiInf} are non-empty and compact, and that the functions $Q$ and $R$ are continuous with respect to all their arguments. They showed that the sequence of solutions obtained for a sequence of finite and countable subsets of $\mathcal{B}$ converges to the solution of \eqref{eqn:SemiInf}. A discussion on convergence rate of local reduction methods also follows\cite{adaptive_Seidel2020}. We show in Theorem \ref{thm:Convergence} when the Algorithm \ref{alg:LocalReductionGeneral} solves problem~\eqref{eq:Plifted}.

\begin{thm}
\label{thm:Convergence}
The solution $(\mathbf{q}^*,r^*,\gamma^*)$ obtained from Algorithm \ref{alg:LocalReductionGeneral} for a non-empty and compact set $\mathbb{W}^N\times\mathbb{D}$ converges to the solution of \eqref{eq:Plifted} if the set $\mathbb{F}\subset \R^{n_q}\times\R^{n_r}\times\R$ such that $G:\mathbb{F}\times\mathbb{W}^N\times\mathbb{D}\rightarrow \mathbb{R}^{n_g}$ is non-empty and compact.
\end{thm}
\begin{proof}
From Theorem \ref{thm:MainResult}, we have $\theta:=(\mathbf{q},r,\gamma)$ and $\rho:=(\mathbf{w},d)$, $\mathcal{A}:=\mathbb{F}$, $\mathcal{B}:=\mathbb{H}$. In \eqref{eq:Prevised}, we take $Q(\theta):=\gamma$ which is linear and thus continuous. Then we have $R(\theta,\rho):=G(\mathbf{q},r,\gamma,\mathbf{w},d)$ which is continuous because both $\max_{h,k} e_h^{T}g_k(\cdot,\cdot,\cdot,\cdot)$ and $J_N(\cdot,\cdot,\cdot,\cdot)$ are continuous. The proof follows Lemma 2.2 from a previous study\cite{Global_Mitsos2011}.
\end{proof}

Theorem \ref{thm:Convergence} requires the constraints in $G$ to be defined over a compact set $\mathbb{F}\times\mathbb{W}^N\times\mathbb{D}$. As a direct consequence, we obtain Remark \ref{rem:BoundedG}.
\begin{rem}[Boundedness of constraints]
    \label{rem:BoundedG}
     $G_{\max}$ is bounded on $\mathbb{F}\times\mathbb{W}^N\times\mathbb{D}$, i.e.
    \begin{equation}
        \exists m,M\in\R: \forall (\mathbf{q},r,\gamma,\mathbf{w},d)\in \mathbb{F}\times\mathbb{W}^N\times\mathbb{D}:\; m\leq G_{\max}(\mathbf{q},r,\gamma,\mathbf{w},d)\leq M
    \end{equation}
\end{rem}
\begin{proof}
The proof follows directly from the extreme value theorem because $G_{\max}(\mathbf{q},r,\gamma,\mathbf{w},d)$ is continuous over a compact set $\mathbb{F}\times\mathbb{W}^N\times\mathbb{D}$.
\end{proof}
Boundedness of $G$ ensures that the maximization step presented in Algorithm \ref{alg:LocalReductionMaximisation} is well-posed. Section \ref{sec:SimilarityScenarios} will further demonstrate the impact of boundedness on the solution obtained from Algorithm \ref{alg:LocalReductionGeneral}. 

We also note that similarly to Assumption \ref{ass:Bounded}, the requirement of boundedness of $G$ need not imply stability of the dynamics \eqref{eq:InitialDynamics}. In particular, the method can be used for solving finite horizon optimal control problems with unstable linear dynamics affected by uncertainty, as will be demonstrated in Section \ref{sec:NumericalResults}.

\subsubsection{Constraint dropping}
The method presented in Algorithm \ref{alg:LocalReductionGeneral} assumes that the cardinality of the sets $\mathbb{H}_j$ is increasing with $j$, i.e. $\mathbb{H}_j\subset\mathbb{H}_{j+1}$ for all $j$. The increasing cardinality corresponds to an increase in the size of the optimisation problem in line nine in Algorithm \ref{alg:LocalReductionGeneral}. The authors\cite{Infinitely_Blankenship1976} provide additional convexity conditions allowing one to drop elements from the set $\mathbb{H}_j$. In particular, they require \eqref{eqn:SemiinfCost} to be strictly convex with respect to $\theta$ and  $R(\theta)$ to be convex with respect to $\theta$ for $\rho\in\mathcal{B}$. Following Theorem \ref{thm:MainResult}, the conditions provided in previous works\cite{Infinitely_Blankenship1976} correspond to strict convexity of \eqref{eq:RevisedCost}, and convexity of $G(\mathbf{z}(\cdot,\cdot,\mathbf{w},d),\mathbf{w},d,\cdot)$ for any $(\mathbf{w},d)\in\mathbb{W}^N\times\mathbb{D}$. In the current work, we do not assume convexity of \eqref{eq:AllConstraints} and \eqref{eq:RevisedCost} is only convex, not strictly convex, so dropping constraints from the set $\mathbb{H}_j$ does not guarantee convergence of the local reduction algorithm. To enable constraint dropping, we consider a special case of problem \eqref{eq:Plifted} where the cost is independent of the uncertainty:

\begin{subequations}
\label{eq:Pnew}
\begin{align}
 \mathcal{P}_N(\mathbb{H}):\quad \min_{{\substack{\mathbf{q},r\\ \mathbf{x}^i,\mathbf{u}^i,\\\forall i\in\mathbb{J}}}}& \quad J_N(\mathbf{q},r) \label{eq:Costnew}
  \\
 \text{s.t. } 
g_k(x_k^i,u_k^i,w_k^i,d^i)\leq 0,\ &\forall i\in\mathbb{J}, k=0,\ldots,N-1\\
(\mathbf{x}^i,\mathbf{u}^i) = \mathbf{z}(\mathbf{q},r,\mathbf{w}^i,d^i),\ &\forall i\in\mathbb{J}
\end{align}
\end{subequations}
Then we can adjust Algorithm \ref{alg:LocalReductionGeneral} to Problem \eqref{eq:Pnew} to enable dropping constraints. First let us rewrite \eqref{eq:AllConstraints} as:
\begin{equation}
G(\mathbf{x}^i,\mathbf{u}^i,\mathbf{w}^i,d^i):=\max_{h,k}\; e_h^T g_k(x^i_k,u^i_k,w^i_k,d^i)
\label{eq:AllConstraintsNew}    
\end{equation}
and \eqref{eq:Gmax} as:
\begin{equation}
    G_{\max,\text{new}}(\mathbf{q},r,\mathbb{H}):=\max_{
{\substack{(\mathbf{w},d)\in\mathbb{H}\\ (\mathbf{x},\mathbf{u})=\mathbf{z}(\mathbf{q},r,\mathbf{w},d)}}
} G(\mathbf{x},\mathbf{u},\mathbf{w},d)
\end{equation}
where $(\mathbf{q},r)\in\mathbb{F}_{\text{new}}\subset \R^{n_q}\times\R^{n_r}$ are found in a subset of the whole search space. 

\begin{thm}[Adapted\cite{Infinitely_Blankenship1976}]
\label{thm:CstrDropping}
If $J_N(\cdot,\cdot)$ is strictly convex, \eqref{eq:AllConstraintsNew} is convex w.r.t. $x_k^i$ and $u_k^i$ for any $(w^i_k,d^i)\in\mathbb{W}\times\mathbb{D}$, and $\mathbb{F}_{\text{new}}$ is convex, then \eqref{eq:AddScenarios} can be replaced by:
\begin{equation}
      \mathbb{H}_{j+1}\leftarrow\mathbb{H}_{j}\cup (\mathbf{w}^j,d^j)\setminus \mathbb{Z}_j
      \label{eq:ModifiedDropping}
\end{equation}
where
\begin{equation}
    \mathbb{Z}_j:=\lbrace (\mathbf{w},d)\in\mathbb{H}_{j}\mid  G_{\max,\text{new}}(\mathbf{q}^j,r^j,\mathbb{H}_{j})<0 \rbrace
\end{equation}
and the solution of the modified algorithm will converge to the solution of \eqref{eq:Pnew}.
\end{thm}
\begin{proof}
Taking $\theta:=(\mathbf{q},r)$, $\rho:=(\mathbf{w},d)$, $Q(\theta):=J_N(\mathbf{q},r)$, $R(\theta,\rho):=G(\mathbf{z}(\mathbf{q},r,\mathbf{w},d),\mathbf{w},d)$, $\mathcal{A}:=\mathbb{F}_{\text{new}}\subset \R^{n_q}\times\R^{n_r}$ in \eqref{eq:Pnew} we obtain the form of \eqref{eqn:SemiInf}. Then the proof follows from Theorems 2.2--2.4 in \cite{Infinitely_Blankenship1976}.
\end{proof}

\subsection{Inexact local reduction}
To simplify the exact local reduction from Algorithm \ref{alg:LocalReductionGeneral} we propose an inexact formulation of the algorithm focusing on numerical properties of the solvers used for the optimal control problems. The exact algorithm for local reduction presented in Algorithm \ref{alg:LocalReductionGeneral} assumes that the maximisation step finds the global solutions to the maximisation problem and only one scenario obtained in this step is then added to the scenario set. These assumptions are often difficult to satisfy. In practice, there are two possible cases:
\begin{itemize}
\item The maximisation step in a given iteration has multiple solutions in general, but only a limited number is used,
\item The maximisation step is solved approximately.
\end{itemize}
In particular, we will focus on analysing the case when local solvers are used. 

If a global solver is used, but only a limited number of scenarios is added, the local reduction algorithm needs more iterations to find a solution than in the case of adding all the scenarios~\cite{Cutting_Mutapcic2009}. Therefore there exists a trade-off between the speed of convergence of the local reduction method and the size of the problem solved in the minimisation step. We show the impact of approximate solutions by considering \emph{similarity of scenarios}, i.e.\ when the interim worst-case scenarios from line 3 in Algorithm \ref{alg:LocalReductionGeneral} are considered similar.

\begin{defn}[Similar scenarios]
\label{def:ScenarioSame}
Let $(\mathbf{w}^1,d^1)$ and $(\mathbf{w}^2,d^2)$ be two scenarios and let $\epsilon_w\geq 0$ and $\epsilon_d\geq 0$ be fixed parameters. The two scenarios are similar if
\begin{equation}
\frac{1}{N}\|\mathbf{w}^1-\mathbf{w}^2\|_2^2\leq \epsilon_{w}
\label{eq:Normw}
\end{equation} 
and
\begin{equation}
\|d^1-d^2\|_2^2\leq \epsilon_{d}.
\label{eq:Normd}
\end{equation} 

\end{defn}

Using Definition \ref{def:ScenarioSame}, we modify line 7 in Algorithm \ref{alg:LocalReductionGeneral} so that the scenario $(\mathbf{w}^j,d^j)$ in iteration $j$ is added to the current set of scenarios if it is not similar to any of the scenarios in $\mathbb{H}_j$. Algorithm \ref{alg:LocalReductionInexact} summarizes the inexact local reduction method with the evaluation of when the scenarios are similar.

\begin{algorithm2e}[t]
\SetAlgoLined
  \KwInput{Initial guess for $\mathbf{q}$, $r$, $\gamma$, $\mathbb{H}_1\neq\emptyset$, and the tolerances $\epsilon_w$ and $\epsilon_d$}
  \KwOutput{Optimal  $\mathbf{q}^*$, $r^*$, $\gamma^*$, set of scenarios $\mathbb{H}^*$ that includes the worst-case}
  Set $\mathbf{q}^1\leftarrow\mathbf{q}$, $r^1\leftarrow r$, $\gamma^1\leftarrow\gamma$ $j\leftarrow 1$ 

\Repeat{$\operatorname{card}\mathbb{H}_{j}= \operatorname{card}\mathbb{H}_{j-1}$}{
Compute $G_{\max}(\mathbf{q}^j,r^j,\gamma^j,\mathbb{W}^N\times\mathbb{D})$
and a maximizer $(\mathbf{x}^j,\mathbf{u}^j,\mathbf{w}^j,d^j)$ by solving~\eqref{eq:Gmax} with $\mathbb{H}=\mathbb{W}^N\times \mathbb{D}$.

\eIf{$G_{\max}(\mathbf{q}^j,r^j,\gamma^j,\mathbb{W}^N\times\mathbb{D})\leq 0$}{
\begin{equation}
  \mathbb{H}_{j+1}\leftarrow\mathbb{H}_{j}
  \end{equation}
  }{
      \ForAll{$(\mathbf{w},d)\in \mathbb{H}_{j}$}{
      \eIf{$\frac{1}{N}\|\mathbf{w}^j-\mathbf{w}\|_2^2 > \epsilon_w$ \upshape{\textbf{or}} $\|d^j-d\|_2^2 > \epsilon_d$}{
   \begin{equation}
   \mathbb{H}_{j+1}\leftarrow\mathbb{H}_{j}\cup (\mathbf{w}^j,d^j)
   \end{equation}
    }{
\begin{equation}
  \mathbb{H}_{j+1}\leftarrow\mathbb{H}_{j}
  \end{equation}    }
      }

    Find a $(\mathbf{q}^{j+1},r^{j+1},\gamma^{j+1})$ that solves $\mathcal{P}_N(\mathbb{H}_{j+1})$ using~\eqref{eq:Pgamma} or \eqref{eq:Prevised}.
    
} 
      Set $(\mathbf{q}^*,r^*,\gamma^*)\leftarrow(\mathbf{q}^j,r^j,\gamma^j)$  and $\mathbb{H}^*\leftarrow \mathbb{H}_j$.

       Set $j\leftarrow j+1$.

 }
 \caption{Inexact local reduction method\label{alg:LocalReductionInexact}}
\end{algorithm2e}

\subsubsection{Impact of similarity of scenarios}
\label{sec:SimilarityScenarios}
The algorithm provides a solution that is robust to $s$ scenarios in the sense of Definition \ref{def:Robusts}, where $s=\text{card}\; \mathbb{H}^*$. From Remark~\ref{rem:BoundedG} and \eqref{eq:AllConstraints}, we already have that for any two scenarios $(\textbf{w}^1,d^1)$, $(\textbf{w}^2,d^2)$, the constraints $g_k$ are bounded and hence: 
\begin{equation}\| g_k(z_k^1, w^1_{k},d^1_{k})-g_k(z_k^2, w^2_{k},d^2_{k}) \|_2\leq \| g_k(z_k^1, w^1_{k},d^1_{k}) \|_2 + \| g_k(z_k^2, w^2_{k},d^2_{k}) \|_2\leq 2M.
\end{equation}
A tighter bound can be obtained if we use Assumption \ref{ass:Bounded} and note that for a chosen parameterization $\textbf{q}^*,r^*,\gamma^*$ the trajectories $\mathbf{z}$ are BIBO-stable w.r.t. disturbances, i.e.
\begin{equation}
   \forall \varsigma_w >0:  \|(\mathbf{w},d)\|_2\leq \varsigma_w\Longrightarrow \exists \varsigma_z>0: \|\mathbf{z}\|_2\leq \varsigma_z
   \label{eq:BIBO}
\end{equation}
Theorem \ref{thm:CstrSatisfaction} uses \eqref{eq:BIBO} to show the impact of the similarity of scenarios on the constraint satisfaction.

\begin{thm}
\label{thm:CstrSatisfaction}
Let $(\textbf{w}^1,d^1)$, $(\textbf{w}^2,d^2)$ be two identical scenarios with $\epsilon_w=\frac{1}{N}\epsilon_w^*$ and $\epsilon_d=\epsilon_d^*$, and $\|(\textbf{w}^i,d^i)\|_2\leq \varsigma_w$ for $i=1,2$. Let $\textbf{q}^*,r^*,\gamma^*$ be the solution of \eqref{eq:Plifted} obtained for $(\textbf{w}^1,d^1)$ and let $\textbf{z}^1=\textbf{z}(\textbf{q}^*,r^*,\textbf{w}^1,d^1)$, $\textbf{z}^2=\textbf{z}(\textbf{q}^*,r^*,\textbf{w}^2,d^2)$. Then the constraint violation for $(\textbf{w}^2,d^2)$ is bounded and
\begin{equation}
    \| g_k(z_k^1, w^1_{k},d^1_{k})-g_k(z_k^2, w^2_{k},d^2_{k}) \|^2_2\leq L^2\left(\epsilon^*_w+\epsilon^*_d+4\varsigma_z^2\right)
    \label{eq:UpperBound}
\end{equation}
where $z_k^i$ is the trajectory $\textbf{z}^i$ at time $k$, $i=1,2$, and $L$ is a local Lipschitz constant for a given $k$.
\end{thm}

\begin{proof}
We have
\begin{equation}
    \|\textbf{w}^1-\textbf{w}^2\|^2_2\leq \epsilon_w^* 
    \label{eq:TimeUncertainty}
\end{equation}
and 
\begin{equation}
    \|d^1-d^2\|_2^2\leq \epsilon_d^*.
        \label{eq:ConstUncertainty}
\end{equation}
We can write \eqref{eq:TimeUncertainty} as:
\begin{equation}
\begin{aligned}
    \|\textbf{w}^1-\textbf{w}^2\|_2^2=&{}\|(w_1^1-w_1^2,\ldots,w_N^1-w_N^2)\|^2\\
    =&{} \Big\|\sum\limits_{k=1}^Ne_k(w_k^1-w_k^2)\Big\|_2^2\\ 
    \leq&{} \epsilon_w^*\\
    \end{aligned}
\end{equation}
where $e_k$ is the $k^{\text{th}}$ row of an identity matrix $\mathbb{I}_{N} $.
Without loss of generality, we can assume that $w_k^1-w_k^2\neq 0$ for $k=1,\ldots, N$. Then from orthogonality of the set $\lbrace e_k(w_k^1-w_k^2)\rbrace_{k=1,\ldots,N}$, and using \eqref{eq:TimeUncertainty}, we get from the Pythagorean theorem:
\begin{equation}
\epsilon_w^*\geq \Big\|\sum\limits_{k=1}^Ne_k(w_k^1-w_k^2)\Big\|_2^2=\sum\limits_{k=1}^N\|e_k(w_k^1-w_k^2)\|_2^2.
\label{eq:Pitagoras}
\end{equation}
Taking into account that $\|e_k(w_k^1-w_k^2)\|_2^2\geq 0$ for all $k$ and $\|e_k(w_k^1-w_k^2)\|^2=\|(w_k^1-w_k^2)\|_2^2$ in \eqref{eq:Pitagoras}, we get:
\begin{equation}
    \|w_k^1-w_k^2\|_2^2\leq \epsilon_w^*
    \label{eq:RewrittenPitagoras}
\end{equation}
for all $k=1,\ldots,N$. Summing up \eqref{eq:ConstUncertainty} and \eqref{eq:RewrittenPitagoras} gives:
\begin{equation}
     \|w_k^1-w_k^2\|_2^2+ \|d^1-d^2\|_2^2\leq \epsilon_d^*+ \epsilon_w^*
\end{equation}
Then we have:
\begin{equation}
    \|(w_k^1-w_k^2,0)\|_2^2+ \|(0,d^1-d^2)\|_2^2= \|(w_k^1-w_k^2,d^1-d^2)\|_2^2
    \label{eq:PitagorasExtension}
\end{equation}
where $0$ denotes the origin of $\R^{n_d}$ and we used the Pythagorean theorem in $\R^{n_d+1}$. Then we have:
\begin{equation}
     \|(w_k^1-w_k^2,d^1-d^2)\|_2^2= \|(0,w_k^1-w_k^2,d^1-d^2)\|_2^2
     \label{eq:TimeConstUncertain}
\end{equation}
where $0$ denotes the origin of $\R^{n_x+n_u}$.

At the same time, from the assumption that $g_k$ is continuously differentiable with respect to all its arguments, we get the local Lipschitz condition \cite[Ch. 2]{Measure_Wheeden1977}:
\begin{equation}
     \| g_k(z_k^1, w^1_{k},d^1)-{}g_k(z_k^2, w^2_{k},d^2) \|_2 \leq L\|(z^1_k,w_k^1,d^1)-(z^2_k,w_k^2,d^2) \|_2
    \label{eq:Lipschitz}
\end{equation}
where $L$ is the local Lipschitz constant of $g_k$. Taking a square in \eqref{eq:Lipschitz} and using \eqref{eq:TimeConstUncertain}, we obtain:
\begin{equation}
    \begin{aligned}
        \| g_k(z_k^1, w^1_{k},d^1)-g_k(z_k^2, w^2_{k},d^2) \|_2^2 
        \leq &{}L^2\|(z^1_k,w_k^1,d^1)-(z^2_k,w_k^2,d^2) \|_2^2\\
        = &{} L^2\|(z_k^1-z_k^2,0,0)+(0,w_k^1-w_k^2,d^1-d^2)\|_2^2\\
        =&{}  L^2(\|z_k^1-z_k^2\|^2+\|(w_k^1-w_k^2,d^1-d^2)\|_2^2)\\
        \leq &{}  L^2\left((\|z_k^1\|+\|z_k^2\|)^2+\|(w_k^1-w_k^2,d^1-d^2)\|_2^2\right)\\
        \leq &{} L^2\left(\epsilon_d^*+ \epsilon_w^*+4\varsigma_z^2\right)
    \end{aligned}
\end{equation}
which concludes the proof.
\end{proof}
In a similar way the proof can be done for the constraint from \eqref{eq:CstrObjRef} because the cost $J_N$ is also continuously differentiable.

Theorem \ref{thm:CstrSatisfaction} shows that the satisfaction of constraints depends on:
\begin{itemize}
    \item The local Lipschitz constant of the constraints,
    \item The response of the system to disturbances,
    \item The choice of similarity of scenarios.
\end{itemize}
The local Lipschitz constant and the response to disturbances are inherent to the system. We note from Theorem \ref{thm:CstrSatisfaction} that the assumption about having a global solution to the maximisation problem in Algorithm \ref{alg:LocalReductionGeneral} is crucial to ensure no constraint violation. Even if the maximisation step is solved exactly and $\epsilon_w=\epsilon_d=0$, the constraint violation in \eqref{eq:UpperBound} is defined by $4L^2\varsigma_z^2$. At the same time, the choice of parameters in Definition \ref{def:ScenarioSame} affects the constraints violations. 

We also note that \eqref{eq:BIBO} makes no assumptions on $\varsigma_z$. For unstable systems, $\varsigma_z$ may be large, thus making the bound in Theorem~\ref{thm:CstrSatisfaction} uninformative. However, the value of $\varsigma_z$ depends also on the chosen control policy $\pi_k$ which can be used to modify the control invariant set\cite{Set_Blanchini1999} and thus tighten the bound.

Furthermore, the impact of the similarity of scenarios provides information about the solution if local, instead of global, optimization solvers are used to solve the maximisation problems.
To make this precise, let $(\textbf{w}^1,d^1)$ in Theorem \ref{thm:CstrSatisfaction} be a local solution from Algorithm \ref{alg:LocalReductionMaximisation}, and let $(\textbf{w}^2,d^2)$ be a global solution. Taking $\epsilon_w^*$ and $\epsilon_z^*$ such that $\|\textbf{w}^1-\textbf{w}^2\|_2^2\leq \epsilon_w^*$, $\|d^1-d^2\|_2^2\leq \epsilon_d^*$, and $\|(\textbf{w}^i,d^i)\|_2\leq \varsigma_w$ for $i=1,2$, the maximal constraint violation is bounded by \eqref{eq:UpperBound}.

As a direct consequence of Theorem \ref{thm:CstrSatisfaction} we obtain the following result:

\begin{thm}
    \label{thm:MultipleScenarios}
    Let us assume that Algorithm \ref{alg:LocalReductionInexact} finished with scenario set $\mathbb{H}^*$, $\text{card }\mathbb{H}^*=H$. The constraint violation for scenario $(\mathbf{w}^*,d^*)\notin \mathbb{H}^*$ is bounded by \eqref{eq:UpperBound} with $\epsilon_w^*=\max_{i=1,\ldots,H}\|\mathbf{w}^i-\mathbf{w}^*\|^2_2$ and $\epsilon_d^*=\max_{i=1,\ldots,H}\|d^i-d^*\|^2_2$.
\end{thm}
\begin{proof}
    From Theorem \ref{thm:CstrSatisfaction} we have that for all $i=1,\ldots,H$:
\begin{equation}
    \| g_k(z_k^i, w^i_{k},d^i)-g_k(z_k^*, w^*_{k},d^*) \|^2_2\leq L^2\left(\epsilon^*_{w,i}+\epsilon^*_{d,i}+4\varsigma_z^2\right)
\end{equation}
where we assumed that $\|\textbf{w}^i-\textbf{w}^*\|^2_2\leq \epsilon^*_{w,i}$ and $\|d^i-d^*\|^2_2\leq \epsilon^*_{d,i}$.  Thus, we have:
\begin{equation}
\begin{aligned}
    \| g_k(z_k^i, w^i_{k},d^i)-g_k(z_k^*, w^*_{k},d^*) \|^2_2\leq &{}\max_{i=1,\ldots,H}L^2\left(\epsilon^*_{w,i}+\epsilon^*_{d,i}+4\varsigma_z^2\right)\\
    =&{} L^2\left(\max_{i=1,\ldots,H}\epsilon^*_{w,i}+\max_{i=1,\ldots,H}\epsilon^*_{d,i}+4\varsigma_z^2\right)
    \end{aligned}
\end{equation}
Taking in Theorem \ref{thm:CstrSatisfaction} 
\[
\epsilon_w^*=\max_{i=1,\ldots,H}\epsilon^*_{w,i}=\max_{i=1,\ldots,H}\|\mathbf{w}^i-\mathbf{w}^*\|^2_2
\] and 
\[
\epsilon_d^*=\max_{i=1,\ldots,H}\epsilon^*_{d,i}=\max_{i=1\,\ldots,H}\|d^i-d^*\|^2_2
\]
concludes the proof.
\end{proof}

\subsubsection{Impact of number of scenarios}

The bounds obtained in Theorem \ref{thm:MultipleScenarios} allow inference about constraint violation in a practical implementation of Algorithm \ref{alg:LocalReductionInexact}. In particular, we will now analyse how fixing the number of scenarios $H$ affects the constraint violation. For simplicity, we focus on the constraint from \eqref{eq:PaugIneq} but an analogous reasoning can be done for \eqref{eq:CstrObjRef}. Let us assume that the function $g_k$ from \eqref{eq:PaugIneq} is such that:
\begin{equation}
 g_k:\R^{n_x}\times\R^{n_u}\times\R^{n_w}\times\R^{n_d}\rightarrow \mathbb{Q}_k\subset\R^{n_g} 
 \label{eq:ConstraintLeb}
\end{equation}
where $\mathbb{Q}_k$ is Lebesgue-measurable. The Lebesgue measure of the set will be called its \textit{n-dimensional volume}\cite[Ch. 21]{Handbook_Schechter1997} and denoted $\text{Vol}\;\mathbb{Q}_k =\Omega_k$. We assume that the values of $g_k$ are from a uniform distribution over $\mathbb{Q_k}$ for any realization of uncertainty. Let us also introduce for every scenario $(\mathbf{w}^i,d^i)\in \mathbb{H}^*$ the following subsets:
\begin{equation}
    \mathbb{S}_{k,\delta}^i:=\lbrace \overline{g}_k\in\mathbb{Q}_k:\|\overline{g}_k-g_k(z_k^i, w^i_{k},d^i)\|_2\leq \delta \rbrace
    \label{eq:nBall}
\end{equation}
where $\delta\geq 0$ is a constant such that $\mathbb{S}_{k,\delta}^i\subseteq\mathbb{Q}_k$. Then we can introduce
\begin{equation}
    \mathbb{S}_{k,\delta}:=\bigcup\limits_i^{H_q}  \mathbb{S}_{k,\delta}^i
    \label{eq:BigS}
\end{equation}
such that $\mathbb{S}_{k,\delta}^i\neq \mathbb{S}_{k,\delta}^j$ for $i,j=1,\ldots,H_q$, $i\neq j$. The set from \eqref{eq:BigS} collects $H_q$ sets $\mathbb{S}_{k,\delta}^i$ corresponding to distinct values of constraints $g_k(z_k^i, w^i_{k},d^i)$ obtained from scenarios $(w^i_{k},d^i)$ and trajectories $z_k^i$. We have that $H_{\delta}\leq H$, with strict inequality if at least two different realizations of scenarios $(w^{i}_{k},d^{i})$, $(w^{j}_{k},d^{j})$ and the corresponding trajectories $z_k^{i}$, $z_k^{j}$, $i\neq j$, $i,j=1,\ldots,H$, lead to the same  value of the constraint, i.e. $g_k(z_k^{i}, w^{i}_{k},d^{i})=g_k(z_k^{j}, w^{j}_{k},d^{j})$. 
The n-dimensional volume of $\mathbb{S}_{k,\delta}$ is now:
\begin{equation}
    \text{Vol}\; \mathbb{S}_{k,\delta} = \text{Vol} \;\bigcup\limits_i^{H_q}  \mathbb{S}_{k,\delta}^i
    \label{eq:SetDiff}
\end{equation}
From countable subadditivity of Lebesgue measure\cite{Real_Royden2010}, we get $\text{Vol}\bigcup\limits_i  \mathbb{S}_{k,\delta}^i\leq \sum\limits_i  \text{Vol}\;\mathbb{S}_{k,\delta}^i$, and then: 
\begin{equation}
    \text{Vol}\; \mathbb{S}_{k,\delta}^1\leq \text{Vol}\; \mathbb{S}_{k,\delta} \leq \sum\limits_i^{H_q}  \text{Vol}\;\mathbb{S}_{k,\delta}^i
    \label{eq:BoundedVolume}
\end{equation}
Noticing that $\mathbb{S}_{k,\delta}^i$ is a ball in $\R^{n_g}$ with radius $\delta$, we have\cite[p. 135]{Introduction_Sommerville2020}:
\begin{equation}
    \text{Vol}\; \mathbb{S}^i_{k,\delta}= \frac{\pi^{\frac{n_g}{2}}\delta^{n_g}}{\Gamma\left(\frac{n_g}{2}+1\right)}
\end{equation}
where $\Gamma$ is the gamma function, and then:
\begin{equation}\label{eq:volum_rel}
        \frac{\pi^{\frac{n_g}{2}}\delta^{n_g}}{\Gamma\left(\frac{n_g}{2}+1\right)} \leq \text{Vol}\;\mathbb{S}_{k,\delta} \leq \frac{H_{\delta}\pi^{\frac{n_g}{2}}\delta^{n_g}}{\Gamma\left(\frac{n_g}{2}+1\right)}.
\end{equation}
    We can now state:
\begin{thm}
    \label{thm:DependenceonH}
    Let us assume that Algorithm \ref{alg:LocalReductionInexact} is used to solve \eqref{eq:Plifted} where $g_k$ satisfies \eqref{eq:ConstraintLeb} for all $k=0,\ldots,N-1$ with $\text{Vol}\; \mathbb{Q}_k=\Omega_k$, and the values of $g_k$ are uniformly distributed over $\mathbb{Q}_k$. Let us assume that Algorithm \ref{alg:LocalReductionInexact} finished with scenario set $\mathbb{H}^*$, $\text{card }\mathbb{H}^*=H$. Let us assume that a threshold $\delta\in\mathbb{R}$ has been chosen so that the sets $\mathbb{S}_{k,\delta}^i$ from \eqref{eq:nBall} satisfy $\mathbb{S}_{k,\delta}^i\subset \mathbb{Q}_k$ for all $k=0,\ldots,N-1$, $i=1,\ldots,H_q$. The probability that the constraint violation exceeds the threshold $\delta$ for scenario $(\mathbf{w}^*,d^*)\notin \mathbb{H}^*$such that $g_k(z_k^*, w^*_{k},d^*)=\overline{g}^*_k$ is bounded as follows:
   \begin{equation}\label{prob_rel}    
 1-\frac{\pi^{\frac{n_g}{2}}\delta^{n_g}}{\Omega_k\Gamma\left(\frac{n_g}{2}+1\right)}\underset{\raisebox{.5pt}{\textcircled{\raisebox{-.9pt} {1}}}}{\geq} P\left( \overline{g}^*_k \in \mathbb{S}^c_{k,\delta}
 \right)\underset{\raisebox{.5pt}{\textcircled{\raisebox{-.9pt} {2}}}}{\geq} 1-\frac{H_{\delta}\pi^{\frac{n_g}{2}}\delta^{n_g}}{\Omega_k\Gamma\left(\frac{n_g}{2}+1\right)}
\end{equation}
 where $\mathbb{S}^c_{k,\delta}$ is the complement of the set $\mathbb{S}_{k,\delta}$, i.e. $\mathbb{S}^c_{k,\delta}:=\mathbb{Q}_k\setminus \mathbb{S}_{k,\delta}$.   
\end{thm}
\begin{proof}
Dividing by $\Omega_k>0$ all the terms in \eqref{eq:volum_rel} we obtain 
\begin{equation}\label{prob_rela}
\frac{\pi^{\frac{n_g}{2}}\delta^{n_g}}{\Omega_k\Gamma\left(\frac{n_g}{2}+1\right)} \leq 
P\left( \overline{g}^*_k \in \mathbb{S}_{k,\delta}
\right)
\leq \frac{H_{\delta}\pi^{\frac{n_g}{2}}\delta^{n_g}}{\Omega_k\Gamma\left(\frac{n_g}{2}+1\right)}
\end{equation}
because $P\left( \overline{g}^*_k \in \mathbb{S}_{k,\delta}\right) := \frac{\text{Vol}\;\mathbb{S}_{k,\delta} }{\Omega_k} $. Substituting $P\left( \overline{g}^*_k \in \mathbb{S}_{k,\delta}
\right)=1-P\left( \overline{g}^*_k \in \mathbb{S}^c_{k,\delta} \right)$ into \eqref{prob_rela} and rearranging, we obtain \eqref{prob_rel}.
\end{proof}

\begin{figure}[tbp]
\centering
\includegraphics[width=0.75\textwidth]{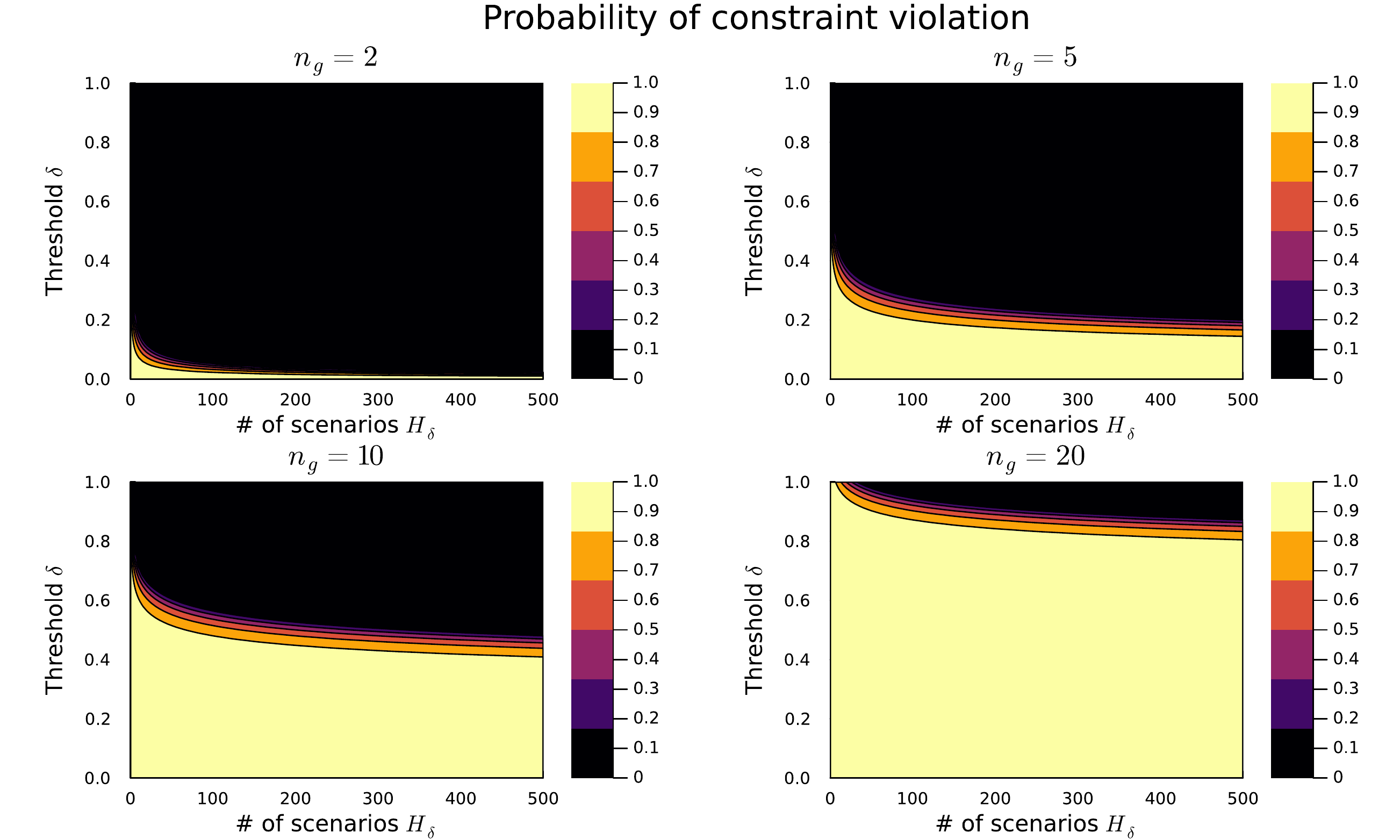}
\caption{Trade-off between the number of scenarios $H_{\delta}$ and the desired threshold $\delta$ as a function of dimensionality $n_g$ of the uncertainty space}
\label{fig:ThresholdTradeOff}
\end{figure}

Theorem \ref{thm:DependenceonH} combines the number of scenarios $H_{\delta}$ obtained from local reduction with the threshold for constraint violation $\delta$. Assuming that $g$ has been normalised so that $\Omega_k=1$, we rewrite \eqref{prob_rel} as:
\begin{equation}
\label{eq:FinalInequalityRewritten}
      P\left( \overline{g}^*_k \in \mathbb{S}^c_{k,\delta} \right)\geq 1-\kappa H_{\delta}\delta^{n_g}
\end{equation}
where $\kappa:=\frac{\pi^{\frac{n_g}{2}}}{\Gamma\left(\frac{n_g}{2}+1\right)}$ is a constant depending on the number of constraints $n_g$. We see in \eqref{eq:FinalInequalityRewritten} that there is a trade-off between the threshold $\delta$ and the number of scenarios $H_{\delta}$. For a given threshold $\delta$, increasing the number of scenarios will decrease the probability of violation. Conversely, if the number of scenarios is constant, a small threshold may be exceeded with high probability. An example of the trade-off is shown in Fig.\ \ref{fig:ThresholdTradeOff} as a function of the dimensions $n_g$.

Finally, Theorem \ref{thm:DependenceonH} has a practical interpretation. Assuming that the allowed probability of constraint violation is $p_{\text{des}}$, from \eqref{eq:FinalInequalityRewritten} we obtain a bound on the number of scenarios $H_{\delta}$ to satisfy this requirement for a given $\delta$:
\begin{equation}
    H_{\delta}\geq \frac{1-p_{\text{des}}}{\kappa\delta^{n_g}}. 
    \label{eq:MinScenarios}
\end{equation}
The number of scenarios obtained from \eqref{eq:MinScenarios} can be used as an additional stopping criterion (line 18 in Algorithm \ref{alg:LocalReductionInexact}).

Finally, should the probability distributions of the uncertainty be available, our approach can be extended to explicitly include chance constraints and probability distributions a priori in the problem formulation by using an appropriate problem structure\cite{Theory_Bertsimas2011}.

\section{Examples}
\label{sec:NumericalResults}
We will first show that the realisation of uncertainty leading to the maximal constraint violation can be anywhere, not necessarily on the boundary of the uncertainty set even for a linear system. We also show the performance of the local reduction applied to an unstable system.

We then show that the local reduction method described in Section \ref{sec:LocalReduction} finds scenarios from inside the uncertainty sets and provides robust solutions to optimal control problems with uncertainty in two numerical examples: temperature control in a residential building and flow control in a centrifugal compressor. The solution provided by local reduction is then compared with the case obtained for boundary scenarios and for scenarios chosen randomly from a uniform distribution. 

The examples were implemented in Julia 1.6 \cite{Julia_Bezanson2017} using JuMP 0.21.4 \cite{JuMP_Dunning2017}. The problems were then solved using Ipopt version 3.12.10 \cite{Exact_Thierry2020}. All tests were performed on an Intel\textsuperscript{\textregistered} Core\textsuperscript{\texttrademark} i7-7500U with 16\,GB of RAM.

\subsection{Scenarios not on the boundary}
\label{prop:NotBoundary}
An example of scenarios not on the boundary for a nonlinear system was provided in previous works\cite{Interval_Krasnochtanova2010,Robust_Puschke2018}. We show that a linear system with parametric uncertainty may have interim worst-case scenario in the interior of the uncertainty range. The worst-case scenario in the sense of Definition \ref{def:WorstCase} for a robust optimal control problem of the form \eqref{eq:Prevised} may be in the interior of the set $\mathbb{W}\times\mathbb{D}$. 

Let us assume that we have a system with dynamics affected by parametric uncertainty $d$:
\begin{equation}
x_{k+1}=(A+d)x_k+Bu_k
\label{eq:UncertainDynamics}
\end{equation}
where $k=1,\ldots,N$ and $d\in[\underline{d},\overline{d}]$, $A$, $B$ are constant scalar matrices, $x_0$ is known. Let us assume further that the optimal control problem includes a constraint of the form:
\begin{equation}
x_k\leq 0
\label{eq:SimpleCstr}
\end{equation}
The constraint \eqref{eq:SimpleCstr} must be satisfied for all $k$. The maximisation step in the local reduction method consists in solving a series of optimisation problems with the objective for every~$k$:
\begin{align}
    \max_{d} \quad x_k\\
\text{ subject to }     \eqref{eq:UncertainDynamics}
\label{eq:MaxOptTest} 
\end{align}
Every $k$ corresponds to a different optimisation problem of the form \eqref{eq:MaxOptTest}. 

Let us now take $k=4$, $x_0=0$, $A=-0.5$, $B=1$, and we are looking for the maximal constraint violation for a constant $d\in[-0.5,0.5]$. Let us assume that the current optimal control input $u_k$, $k=0,\ldots,4$ is $u_0=u_2=u_3=-1$ and $u_1=u_4=1$.  The maximisation problem from \eqref{eq:MaxOptTest} becomes a maximisation of a fourth order polynomial of $d$:
\begin{equation}
\max_{d}\sum\limits_{j=0}^4u_{4-j}(-0.5+d)^j
\label{eq:SpecialPoly}
\end{equation}
Looking for the maximum of the polynomial \eqref{eq:SpecialPoly} yields $d\approx 0.2$ which is not on the boundary of the interval $[-0.5,0.5]$.

This example confirms that considering boundary scenarios would miss the actual worst-case scenario. In a similar way, selecting a priori a number of scenarios would result in adding unnecessary scenarios that may or may not be the worst-case scenario. 
\subsection{Unstable system}
To show that the boundedness of constraints required in Theorem \ref{thm:Convergence} need not imply stability, we analysed a system with dynamics:
\begin{equation}
    x_{k+1}=adx_{k}+u_k
    \label{eq:Unstable}
\end{equation}
where $a=2.1$, the uncertainty $d\in[0.9,1.1]$, and with $x_1=0.5$. The constraints to be satisfied were $0\leq x_k\leq 1$, for all $k=1,\ldots,N$ with $N=10$. The controller $u_k$ was parameterised as an affine function of the state:
\begin{equation}
    u_k=Kx_k+q_k
\end{equation}
and the constraints $u_k\in[-1,1]$ were enforced by a smooth saturation function
\begin{equation}
u^{\text{sat}}_k=\frac{\beta_0}{\beta_1+\exp(\beta_2u_k)}+\beta_3
\end{equation}
where $\beta_i$ are constants. Here $\beta_0=-2.0229$, $\beta_1=1$, $\beta_2=1.2963$, $\beta_3=1.01145$. The objective was to minimise the square of the control over the whole horizon $N$:
\begin{equation}
J=\sum\limits_{k=0}^{N-1}u_k^2.
\end{equation}

The trajectory obtained from maximization of the violation of the constraint $x_{10}\leq 1$ is shown in Fig. \ref{fig:TrajectoryViolation}. The time horizon is finite, $N=10$, so the trajectory $\textbf{x}$ is bounded for any value of $d$ and thus the local reduction can be used. Using the local reduction algorithm resulted in three scenarios: $d_1=1$, $d_2=0.9$, $d_3=1.1$ that ensure robustness, as indicated in Fig. \ref{fig:UnstablePlot3}. The trajectories in Fig. \ref{fig:UnstablePlot3} were obtained for 500 randomly chosen scenarios uniformly distributed in $[0.9,1.1]$. 

\begin{figure}
     \centering
     \begin{subfigure}[b]{0.47\textwidth}
         \centering
         \includegraphics[width=\textwidth]{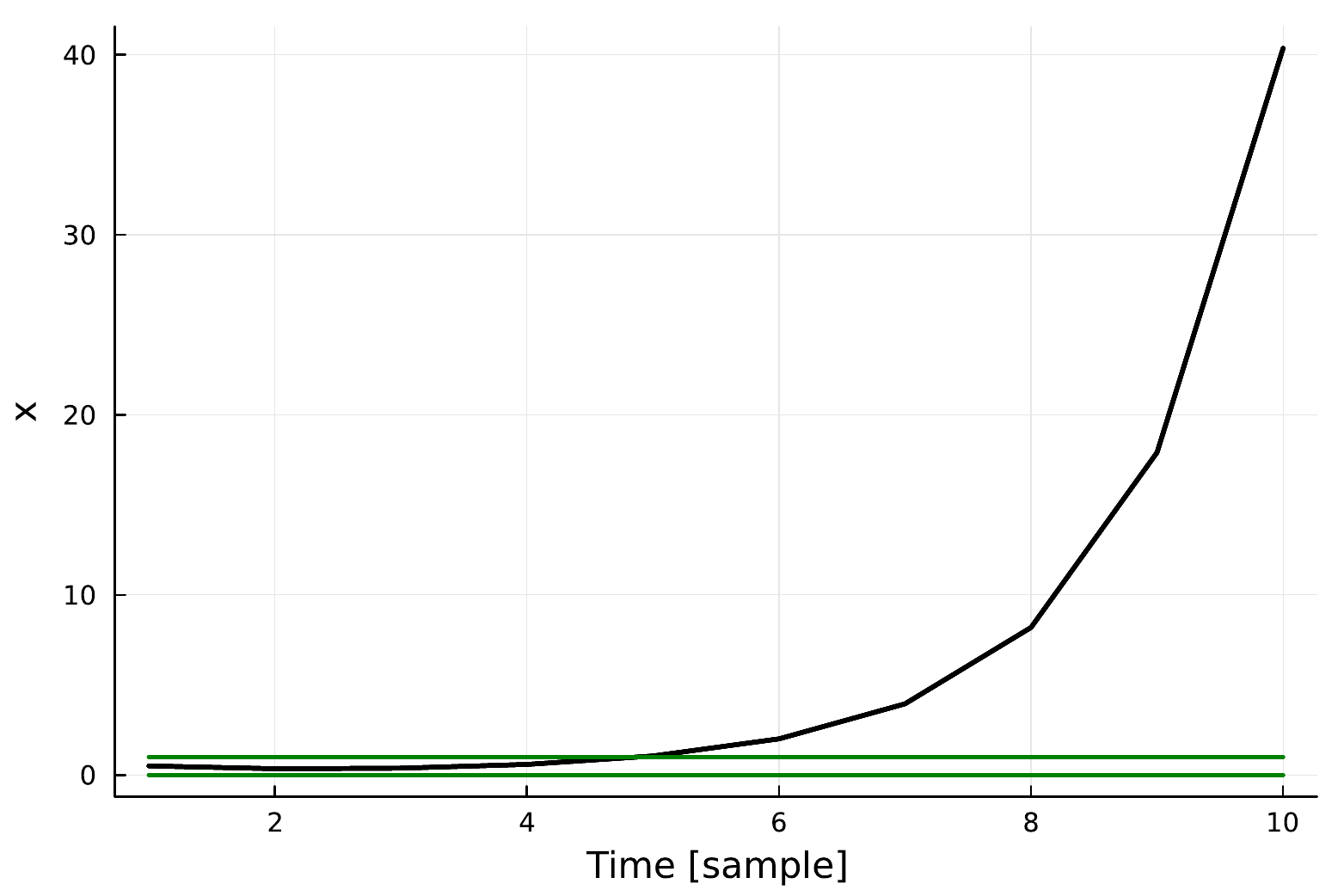}
         \caption{Trajectory $x$ corresponding to maximization of upper bound (black) with the respective bounds (green)}
         \label{fig:TrajectoryViolation}
     \end{subfigure}
     \hfill
     \begin{subfigure}[b]{0.47\textwidth}
         \centering
         \includegraphics[width=\textwidth]{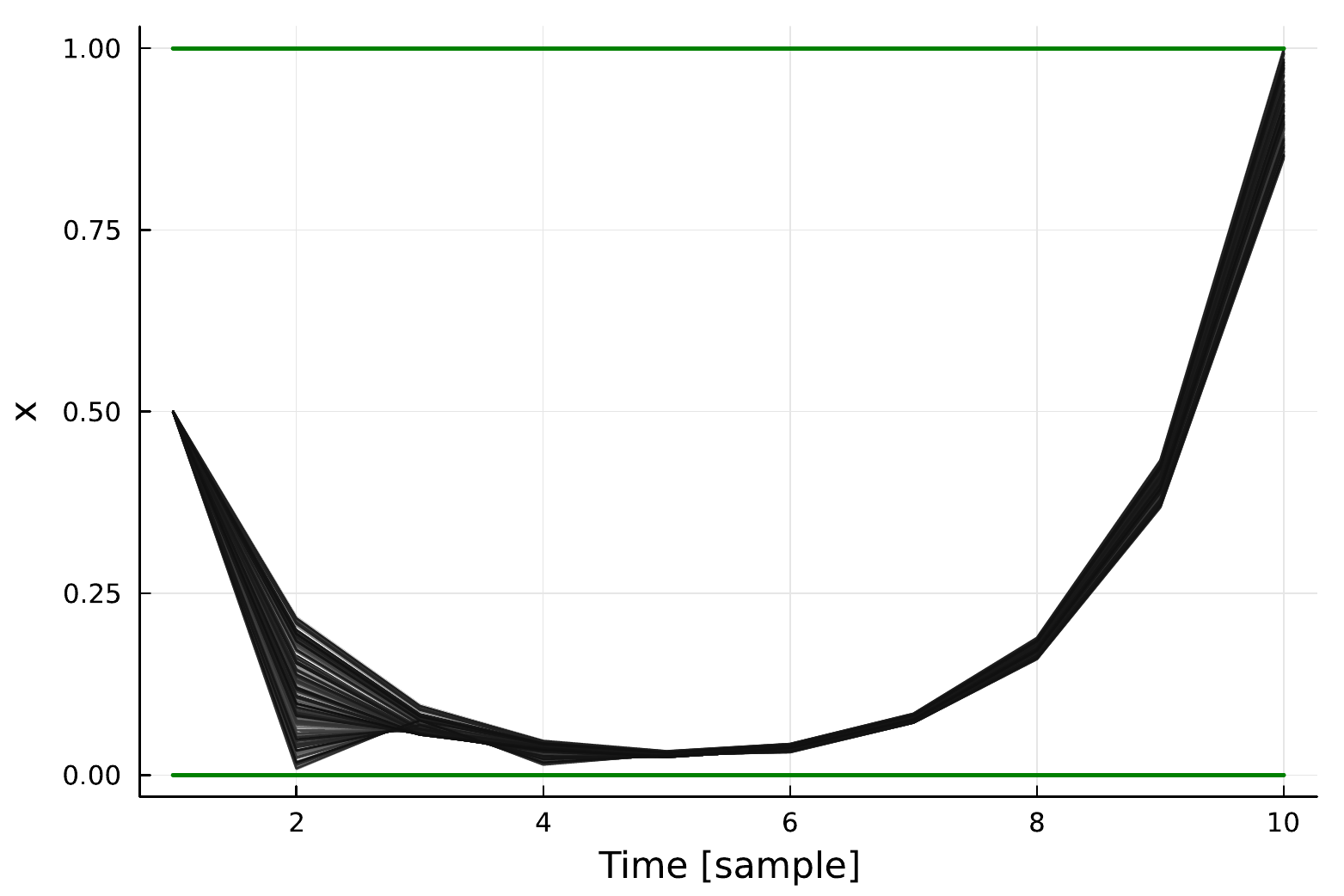}
         \caption{Validation of the scenarios from local reduction for 500 randomly chosen scenarios uniformly distributed in $[0.9,1.1]$}
         \label{fig:UnstablePlot3}
     \end{subfigure}
        \caption{Application of local reduction for the unstable system \eqref{eq:Unstable}}
        \label{fig:UnstableSystem}
\end{figure}

The example shows that stability of the dynamics is optional, provided the constraints are bounded. Nevertheless, it is recommended to take instability into account when  numerically solving the miximization problems in Algorithm \ref{alg:LocalReductionMaximisation}. To show the influence of instability on the numerical performance, we computed the violation for six different values of the horizon $N$ and the results are collected in Table \ref{tbl:Violations}. For instance, the maximization problems can be solved approximately, terminating as soon as a violation has been found. Such approaches correspond to choosing large values of $\epsilon_w$ and $\epsilon_d$ in Algorithm \ref{alg:LocalReductionInexact}, so the analysis in Section \ref{sec:SimilarityScenarios} holds.

\begin{table}[!tbp]
\centering
\caption{Constraint violation obtained for the unstable system \eqref{eq:Unstable} with different time horizon $N$}
\begin{tabular}{@{}l|llllll@{}}
$N$         & 5    & 10    & 15      & 20        & 25        & 30        \\ \midrule
Violation & 3.75 & 39.37 & 49829.0 & 2.63163e6 & 2.00637e8 & 1.35451e9
\end{tabular}
\label{tbl:Violations}
\end{table}

\subsection{Linear system with parametric uncertainty}
This numerical example consists in a linear system with both parametric and additive time-varying uncertainty. The example describes a single zone building affected by time-varying internal heat gain, solar radiation, and external temperature\cite{System_Lian2021}. The objective is to follow a time-varying set-point for internal temperature $x_k^{\text{temp}}$. The dynamics are discrete and linear:
\begin{equation}
x_{k+1}=Ax_k+Bu_k^{\text{sat}}+Ww_k
\label{eq:BuildingSimple}
\end{equation}
with matrices:
\begin{center}
    $A=\begin{bmatrix}
    0.8511 & 0.0541&  0.0707\\
    0.1293 & 0.8635 & 0.0055\\
    0.0989  & 0.0032 & 0.7541
\end{bmatrix}$, $B=10^{-3}\begin{bmatrix}
    3.5\\
    0.3\\
    0.2
\end{bmatrix}$, $W=10^{-3}\begin{bmatrix}
    22.217& 1.7912 &42.2123\\
    1.5376 &0.6944& 2.29214\\
    103.1813& 0.1032& 196.0444
\end{bmatrix}$.
\end{center}
The states $x$ describe the indoor temperature $x^{\text{temp}}$, wall temperature $x^{\text{wall}}$, and the corridor temperature $x^{\text{corr}}$. The control $u$ represents the amount of heating and cooling delivered to the building. The initial condition was chosen as $x_0=\begin{bmatrix} 25.0& 24.0 &24.0 \end{bmatrix}^{\T}$$^{\circ}$C. Moreover, we assume that the wall temperature and the corridor temperature can only be measured approximately, so there are two additional sources of uncertainty in the initial condition for these two states. We assume $x_0^i=24+d^i$, $i=\text{wall}, \text{corr}$, where $d^i\in[-0.5,0.5]$. We also assume that the matrices $A=[a_{i,j}]$ and $B=[b_j]$, $i,j=1,2,3$ are affected by uncertainty:
\begin{center}
        $a_{i,j}\cdot\delta_{i,j}$ and $b_j\cdot\eta_j$
\end{center}
where $\delta_{i,j},\eta_j$ are uncertain parameters. Two cases will be considered: Case A with $\delta_{i,j},\eta_j\in[0.98,1.02]$ and Case B with $\delta_{i,j},\eta_j\in[0.96,1.03]$.
The minimal control effort is ensured by the objective function:
\begin{equation}
J=\frac{1}{N}\sum\limits_{k=0}^{N-1}u_k^2
\end{equation}

It is assumed that the day starts at 6.00$\,$am and lasts 12~hours. The temperature indoors must stay within limits:
\begin{equation}
T_{\min}\leq x_k^{\text{temp}}\leq T_{\max}
\label{eq:TempBounds}
\end{equation}
During the day, the indoor temperature must be kept above $23^\circ$$\,$C and during the night can drop down to $17^\circ$$\,$C:
\begin{equation}
T_{\min} = \begin{cases}
17^{\circ}$\,$\text{C during night time}\\
23^{\circ}$\,$\text{C during day time}
\end{cases}
\end{equation}
 The maximal temperature is the same during the day and night, $T_{\max}=26^{\circ}$$\,$C. 

\begin{table}[]
\centering
\caption{Ranges of uncertain parameters throughout the day}
\label{tbl:Uncertainties}
\begin{tabular}{@{}lll@{}}
\toprule
                     & Day              & Night            \\ \midrule
Internal heat gain   & [4,6]            & [0,2]            \\
Solar radiation      & [4,6]            & 0                \\
External temperature & [6,8]$^{\circ}$C & [2,4]$^{\circ}$C \\ \bottomrule
\end{tabular}
\end{table}

The optimal control problem is solved over a period of 48\,hours starting at 6.00$\,$am the first day, with $N=192$. As a result, the trajectory constraints \eqref{eq:TempBounds} impose $192\cdot 2$ constraints corresponding to every sampling time. The three uncertain parameters, internal gain, solar radiation, and external temperature, vary with time within the limits provided in Table \ref{tbl:Uncertainties}.

The control variables are parameterised as:
\begin{equation}
u_k=Kx_k^{\text{temp}}+q_k
\label{eq:ControlBuilding}
\end{equation}
where $K$ and $q_k$ are decision variables. Furthermore, we include saturation of the control inputs:
\begin{equation}
u^{\text{sat}}_k = \text{sat}(u_k) = \begin{cases}
-500\text{\,W}&{}\text{ for } u_k<-500\text{\,W}\\
u_k&{}\text{ for } -500\text{\,W}\leq u_k\leq1200\text{\,W}\\
1200\text{\,W}&{}\text{ for } u_k>1200\text{\,W}
\end{cases}
\end{equation}
The saturation was approximated by a smooth function:
\begin{equation}
u^{\text{sat}}_k=\frac{\beta_0}{\beta_1+\exp(\beta_2u_k)}+\beta_3
\label{eq:SmoothSaturation}
\end{equation}
where $\beta_i$ are constants. Here $\beta_0=-5030$, $\beta_1=2.937$, $\beta_2=0.003$, $\beta_3=1207$.

In total, there are 14 uncertain parameters affecting the matrices $A$, $B$, and the initial condition for the wall and corridor temperatures. We assume no knowledge about the scenarios, except the ranges of uncertainty. 

\subsubsection{Results - Case A}
In case A, we ran Algorithm \ref{alg:LocalReductionGeneral} with the time-varying uncertainties from Table \ref{tbl:Uncertainties} and parametric uncertainties $\delta_{i,j},\eta_j\in[0.98,1.02]$. 

\paragraph{Overall performance}
The local reduction method in Case A reduced the number of scenarios to five. The resulting controller obtained for the interim worst-case scenarios was then validated for 500 random realisations from a uniform distribution of uncertainty. The validation of the controller is shown in top left plot in Fig. \ref{fig:Comparison}. The black curves stay within the green bounds corresponding to constraints \eqref{eq:TempBounds}. The results suggest that local reduction was able to find a robust solution despite using a local solver for maximisations.

\begin{figure*}
     \centering
     \begin{subfigure}[b]{0.45\textwidth}
         \centering
         \includegraphics[width=\textwidth]{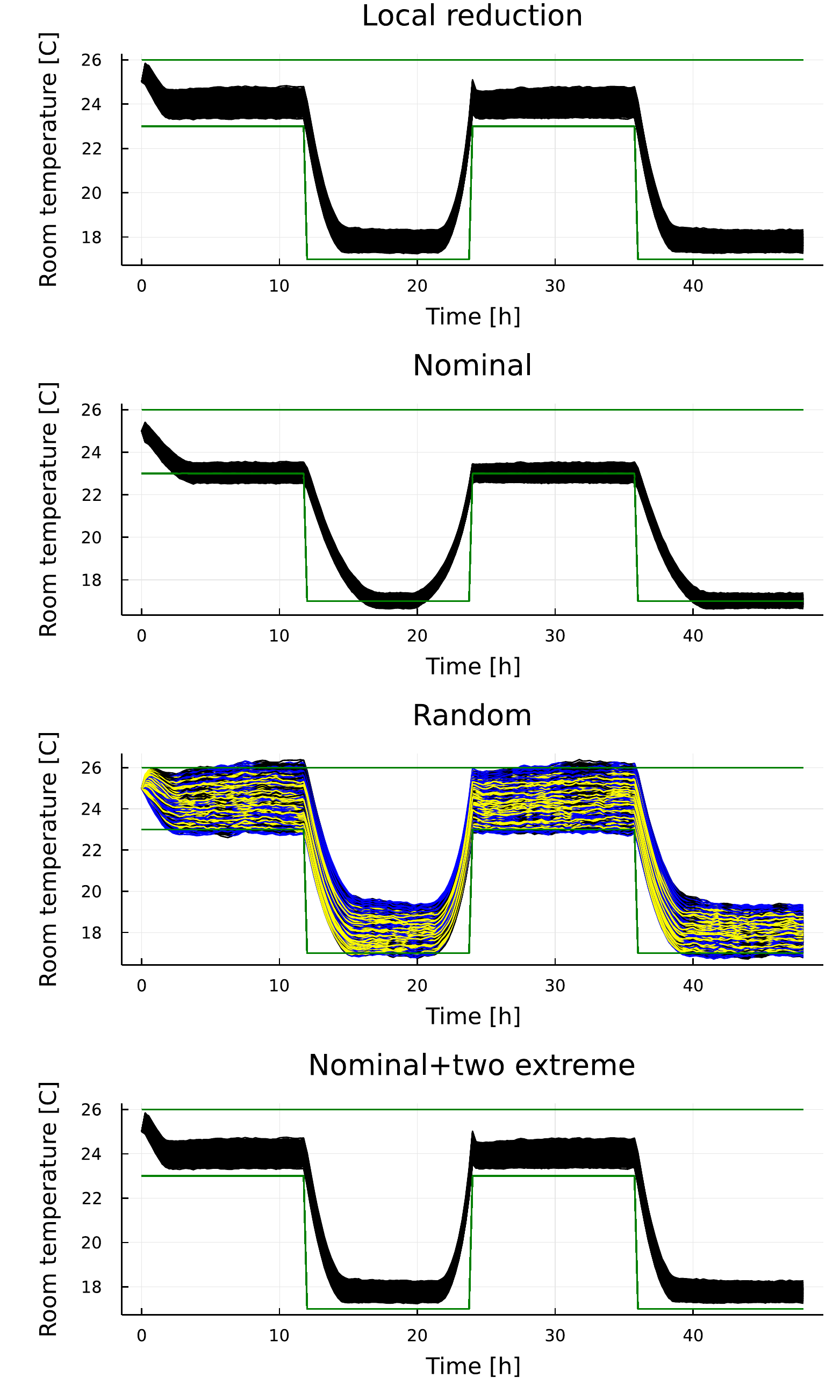}
         \caption{Case A}
         \label{fig:SmallUncertainty}
     \end{subfigure}
     \hfill
     \begin{subfigure}[b]{0.45\textwidth}
         \centering
         \includegraphics[width=\textwidth]{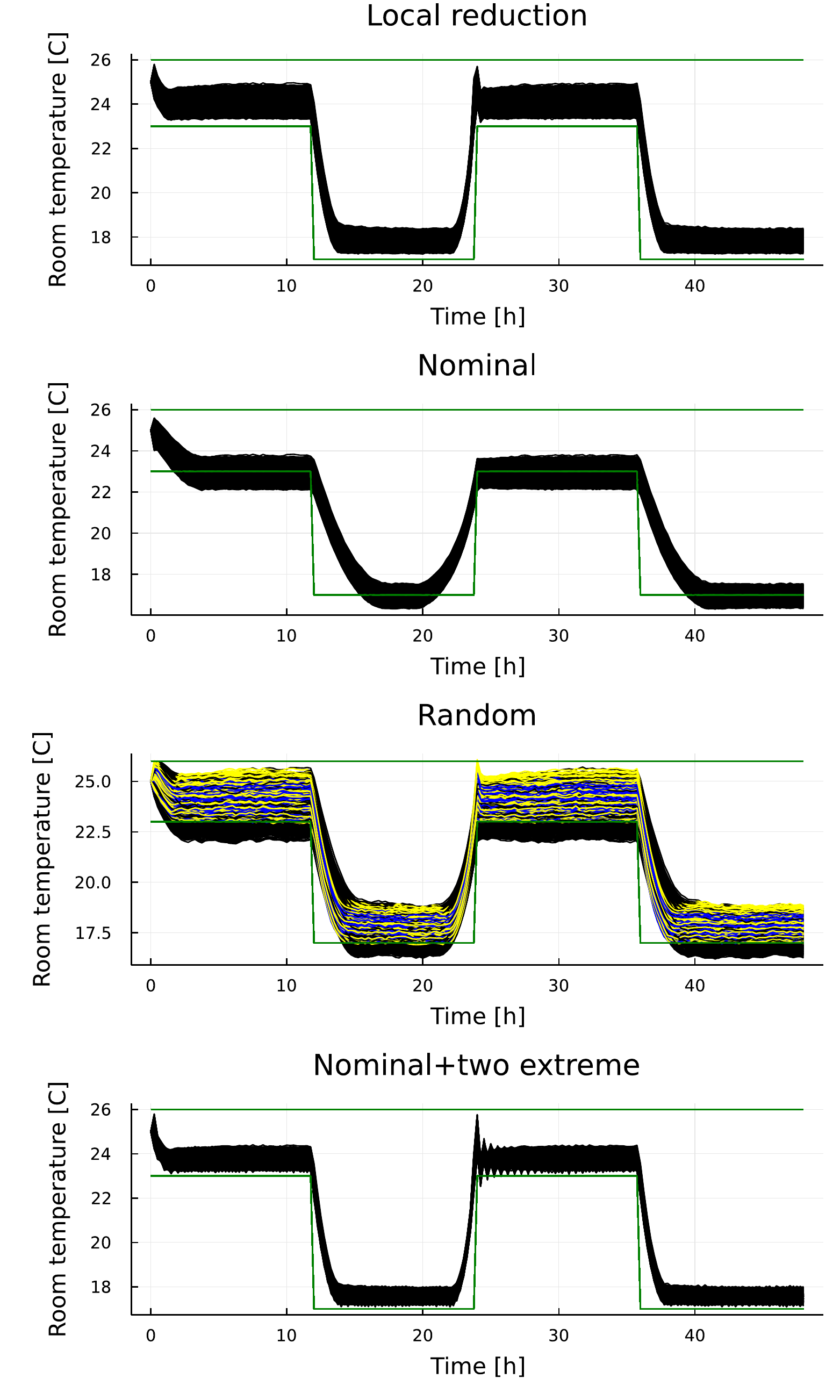}
         \caption{Case B}
         \label{fig:BigUncertainty}
     \end{subfigure}
        \caption{Comparison of local reduction with scenario based approaches in Case A (left) and Case B (right)}
        \label{fig:Comparison}
\end{figure*}

The results also indicate that the local reduction method handles time-varying uncertainty without specifying the scenarios over the whole time horizon. This is because there is no need to specify time-varying scenarios as they will be found in the maximisation step in Algorithm \ref{alg:LocalReductionMaximisation}. Moreover, Algorithm \ref{alg:LocalReductionMaximisation} treats time-varying uncertainty as one realisation over the whole horizon, thus overcoming the limitations of separate robust horizon\cite{Sensitivity_Thombre2021}. 

An example of an interim worst-case scenarios obtained in the maximisation step is shown in Fig. \ref{fig:NonBoundaryWorstCase}.

\begin{figure}[tbp]
\centering
\includegraphics[width=0.75\textwidth]{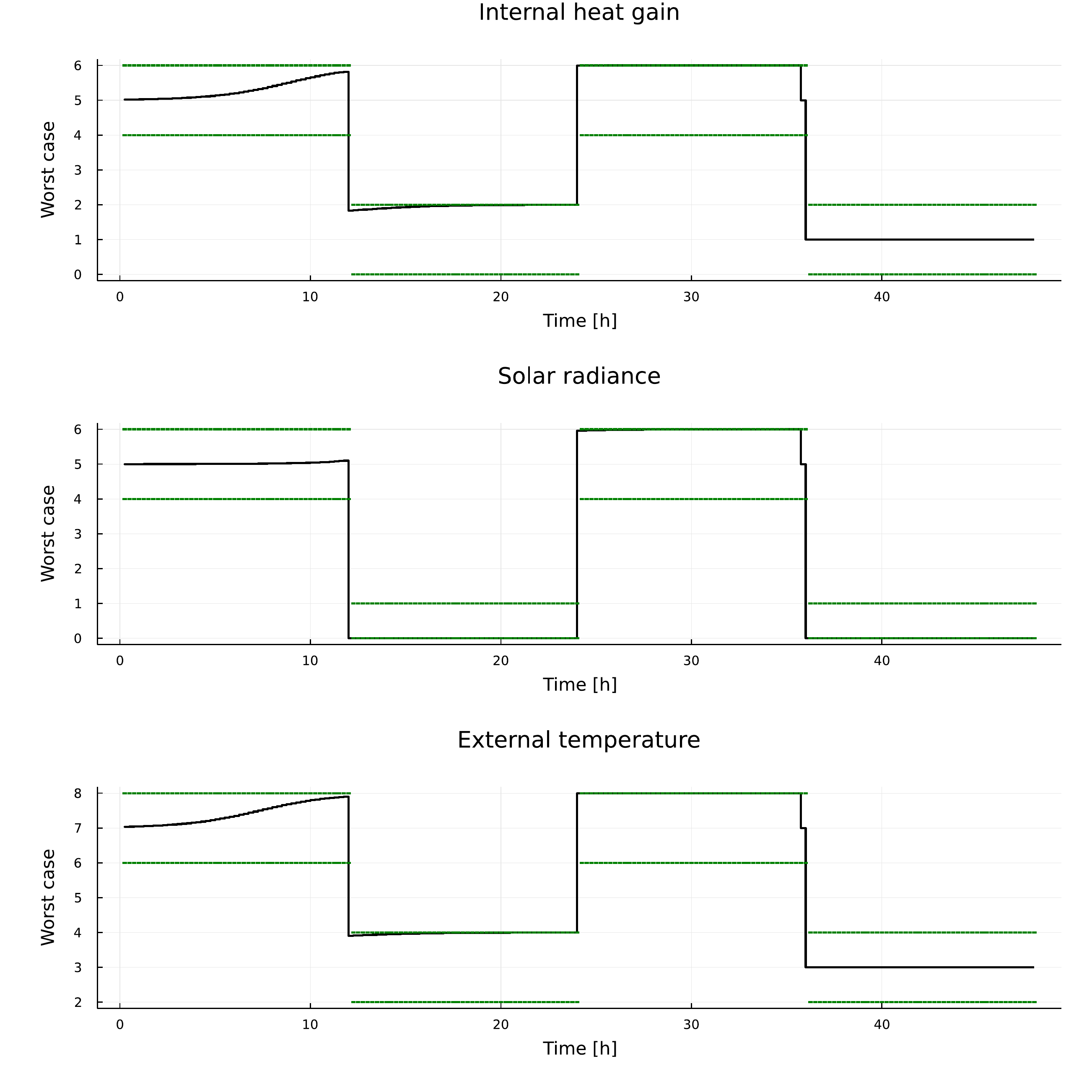}
\caption{An example of a time-varying scenario found by the local reduction in Case A}
\label{fig:NonBoundaryWorstCase}
\end{figure}

\paragraph{Comparison with other approaches}

The results obtained from local reduction are then compared with three scenario-based approaches from the literature \cite{scenario_Campi2021}:
\begin{itemize}
    \item Nominal approach, with a controller obtained assuming there is no uncertainty, i.e. $(\mathbf{w},d)=0$ (further denoted as ``Nominal'')
    \item Randomised approach, with a controller obtained for a number of randomly chosen scenarios (further denoted as ``Random'')
    \item Extreme approach, with a controller obtained for three scenarios: nominal, lower bound, and upper bound for all uncertainties\cite{Sensitivity_Thombre2021} (further denoted as ``Nominal+two extreme'')
\end{itemize}

Validating the nominal controller with 500 random scenarios shows that the approach based on nominal values leads to violation of constraints as shown in the plot `Nominal' in the left column of Fig. \ref{fig:Comparison}.

 The second set of controllers we used was derived using three sets of random scenarios: five scenarios because five scenarios were found in local reduction, 100 scenarios, and 250 scenarios. The results are shown in the plot `Random' in the left column of Fig. \ref{fig:Comparison}, with black corresponding to the controller obtained from five scenarios, yellow to the controller with 100 scenarios, and blue to the controller with 250 scenarios. In all the cases the controller violated at least one of the bounds (100 scenarios gave 0.2$^\circ$ $\,$C, 250 scenarios gave 0.1$^\circ$ $\,$C), with the controller based on five scenarios violating both the lower and upper bound (1.1$^\circ$C). Even though the violation decreased with increasing the number of scenarios, further increasing the number of random scenarios to 600 proved unsuccessful in avoiding the violation. Larger problems could not be solved on the computer.

A possible reason for the random controller being unable to satisfy the constraints is due to not including extreme scenarios in the scenario set. If we were to take only extreme values for every uncertainty and consider all extreme scenarios, we would need to solve a problem with $2^{14+3\times192}$ scenarios, which is intractable. To reduce the number of scenarios, we chose to use the nominal scenario, combined with two extreme scenarios. The extreme scenarios were taken as all uncertainties on their lower or upper bound simultaneously. The results of validating the controller for 500 scenarios are shown in the plot `Nominal+two extreme' in the left column of Fig. \ref{fig:Comparison}. The controller based on the extreme scenarios was also able to avoid constraint violations with three scenarios. 

The results of the comparison show that the local reduction method provides better results than the nominal control or a controller based on a random choice of scenarios. At the same time, the number of scenarios found in the maximisation step from Algorithm \ref{alg:LocalReductionMaximisation} is comparable to the controller based only on nominal and two extreme scenarios. Further analysis of the performance comparison will be done for Case B in Section \ref{sec:CaseB} to show that other controllers fail if parametric uncertainty is more significant.

\paragraph{Time performance}
The results from Fig. \ref{fig:Comparison} indicate that the local reduction method enables reducing the number of scenarios compared to approaches based on random choice or on time-varying extreme scenarios. Figure \ref{fig:TimingBuildings} shows in the left column the time necessary to solve each step of the local reduction method. The plot in the top left shows the time to solve the minimisation problem as a function of iterations of the local reduction. The iterations correspond to the number of scenarios included in the minimisation problem. The number of scenarios is relatively small (five scenarios), so the time for each iteration is below~3.5$\,$s. 

The second plot in the left column shows the time for solving the maximisation problem if the objective is considered (line 2 in Algorithm \ref{alg:LocalReductionMaximisation}). The time to solve the maximisation problem in every iteration is comparable to the time to solve the minimisation problem. The subsequent two plots show the time for the maximisation of the constraint violation corresponding to the lower and upper bounds over the overall time horizon (one maximisation per time step, lines 4-8 in Algorithm \ref{alg:LocalReductionMaximisation}). In both cases, the average time to solve a single maximisation problem was 0.35$\,$s. The total time for finding the five scenarios was 10$\,$min$\,$48$\,$s. The algorithm can be parallelised so that the maximisation problems are solved simultaneously\cite{parallel_Zakovic2003}. Therefore, it can be expected that the time to find a solution in a single iteration of the local reduction method will be equivalent to the solution of the minimisation problem and the maximal time needed to solve the maximisation problems. 

\begin{figure*}[tbp]
\centering
\includegraphics[width=0.9\textwidth]{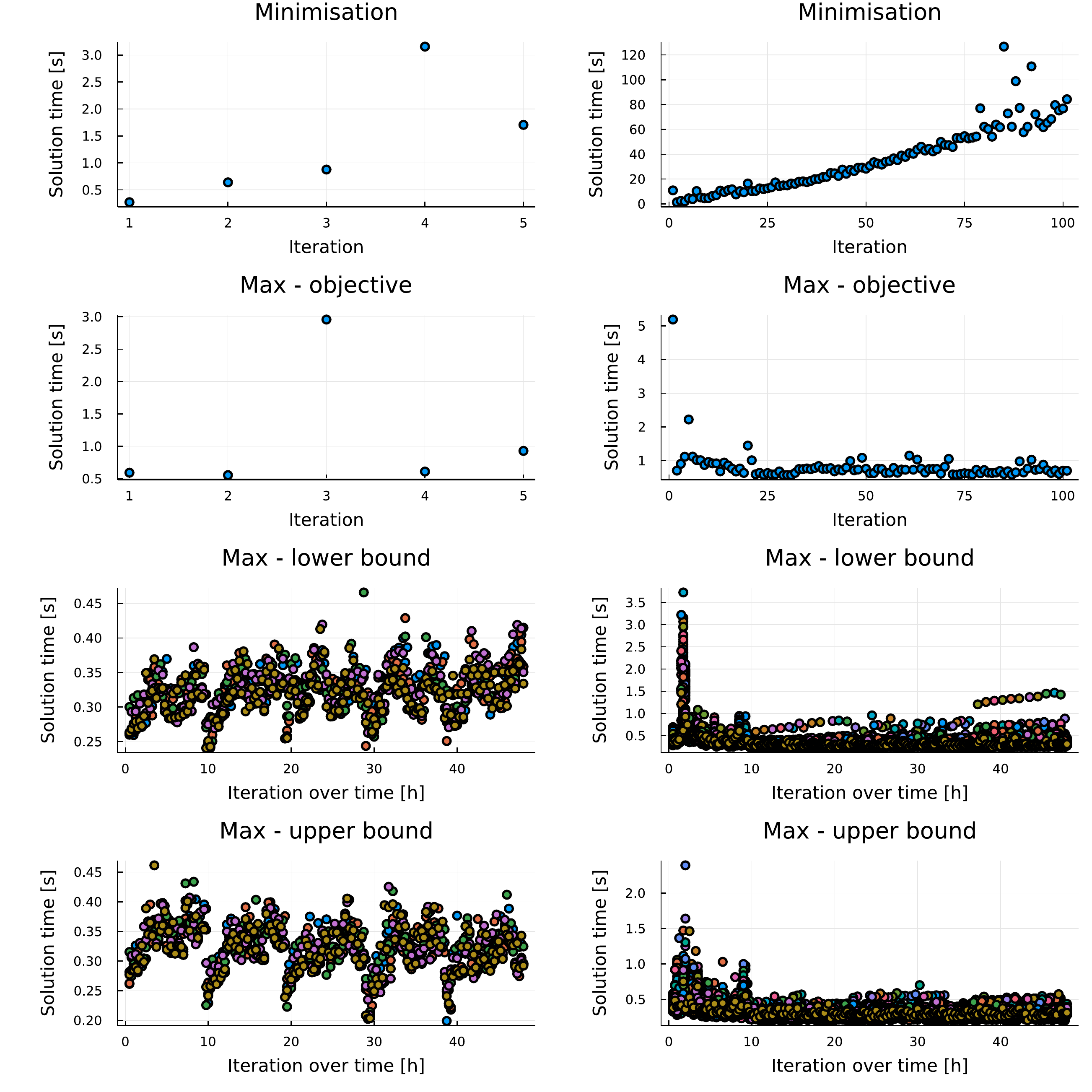}
\caption{Impact of significant parametric uncertainty on local reduction for Case A (left) and Case B (right). The first row shows the time necessary for solving the minimization problem (line 3 in Algorithm \ref{alg:LocalReductionInexact}) as a function of iteration of local reduction, corresponding to the number of scenarios in the current set $\mathbb{H}$. The second row shows the time necessary to solve the first maximization problem related to the objective as a function of iteration of local reduction (line 1 in Algorithm \ref{alg:LocalReductionMaximisation}). The two bottom rows show the time necessary to solve the maximization problems corresponding to $n_g=2$ trajectory constraints \eqref{eq:TempBounds} as a function of samples over the horizon of 48 h, $N=192$, (lines 4-8 in Algorithm \ref{alg:LocalReductionMaximisation})}
\label{fig:TimingBuildings}
\end{figure*}

\subsubsection{Results - Case B}
\label{sec:CaseB}
In case B, we ran Algorithm \ref{alg:LocalReductionGeneral} with the time-varying uncertainties from Table \ref{tbl:Uncertainties} and parametric uncertainties $\delta_{i,j},\eta_j\in[0.96,1.03]$. The range for parameters $\delta_{i,j},\eta_j$ was chosen to increase the parametric uncertainty in the dynamics while still ensuring that a controller of the form~\eqref{eq:ControlBuilding} exists.

\paragraph{Overall performance}
In contrast to Case A, which has given only five scenarios, the local reduction method in Case B found 101 scenarios. A validation for 500 scenarios is shown in the top right plot in Fig. \ref{fig:Comparison}. The plot shows that the controller obtained for 101 scenarios from the local reduction avoided violating constraints. This result indicates that local reduction can handle parametric uncertainty.

\paragraph{Comparison with other approaches}
The controller in Case B has also been compared with the same set of controllers as in Case A, obtained for the new range of uncertainty. As expected, the nominal controller and the random controller were unable to satisfy the constraints (middle plots in the right column in Fig. \ref{fig:Comparison}). In contrast to Case A, the controller based on nominal and two extreme scenarios was also unable to satisfy the constraints, as shown in the bottom right plot in Fig. \ref{fig:Comparison}. The black lines after 24 hours cross the green lines so that the lower bound on the temperature is violated (0.5$^\circ$C). Therefore, taking extreme scenarios may be insufficient, as shown in Section \ref{prop:NotBoundary}. 

The performance of local reduction in handling parametric uncertainty will be also confirmed in the nonlinear case study in Section \ref{sec:Compressors}.

\paragraph{Computational time performance}
Figure \ref{fig:TimingBuildings} shows in the right column the time performance of the elements of the local reduction method in terms of time needed to solve them. As expected, the time to obtain the solution to the minimisation problems increases with iterations. The increase is due to the fact that the number of scenarios considered in every iteration is greater than in the previous one. At the same time, the time necessary to solve a single maximisation problem remained similar across the iterations.  This result indicates the potential for parallelisation to improve performance. 

In this work, we also assume that the structure of the dynamic feedback policy in \eqref{eq:InitialDynamics} is known. The improved computational time of the local reduction can be used to validate whether the chosen control parametrisation  is suitable for robustness, because it enables obtaining a solution more quickly. Thus, if the results for a given parametrisation are unsatisfactory, a different parametrisation can be evaluated.

\subsubsection{Results - properties of interim scenarios}

Finally, we show the impact of the choice of when two scenarios are considered similar in Algorithm \ref{alg:LocalReductionInexact}. The results of varying $\epsilon$ are collected in Table \ref{tbl:similarityCheck}. The time was obtained using BenchmarkTools.jl \cite{Robust_Chen2016}. 

As expected, a high threshold for similarity of scenarios leads to fewer scenarios added to the problem. This is visible in particular in Case B with more significant parametric uncertainty, where the high threshold $\epsilon=0.1$ led to two scenarios, whereas a lower threshold $\epsilon=0.001$ led to 101 scenarios. The middle column in Table \ref{tbl:similarityCheck} shows that robustness to the three scenarios in Case B for $\epsilon=0.1$ is insufficient to robustify the system against random realisations of uncertainties. Conversely, both $\epsilon=0.01$ and $\epsilon=0.001$ seem to robustify the system against the random realisations. Unless explicitly stated, the paper considered $\epsilon=0.001$.

The number of scenarios also affects the time necessary to solve the resulting optimisation problem corresponding to all the scenarios (right column of Table \ref{tbl:similarityCheck}). 

\begin{table}[]
\caption{Influence of the tolerance for checking the similarity of scenarios on the resulting number of scenarios, maximal constraint violation over 500 random scenarios, and the time to obtain a solution for the scenarios obtained}
\label{tbl:similarityCheck}
\centering
\begin{tabular}{lllll}
\toprule
                        & $\epsilon$ & \# scenarios & Max violation & Time     \\ \midrule
\multirow{3}{*}{Case A} & 0.1       & 3            & 0             & 6.4 min  \\
                        & 0.01       & 5            & 0             & 10.8 min \\
                        & 0.001       & 5            & 0             & 10.5 min \\\midrule
\multirow{3}{*}{Case B} & 0.1      & 2            & 1.5 $^\circ$C          &   5.3 min       \\
                        & 0.01       & 101          & 0             & 5 h       \\
                        & 0.001       & 101          & 0             & 5 h      \\ \bottomrule
\end{tabular}
\end{table}

\subsection{Nonlinear system with a dynamic controller}
\label{sec:Compressors}
\subsubsection{Dynamics}
A further case study is presented to show how our proposed method can be used in nonlinear systems. We want to design a flow controller for a centrifugal compressor. The dynamics for a compressor are nonlinear\cite{Experimental_Cortinovis2015}:
\begin{equation}
\begin{aligned}
\dot{p}_{s}&={}\frac{a_{01}^2}{V_{s}}(m_{\text{in}}-m+m_{r}),\\
\dot{p}_{d}&={}\frac{a_{01}^2}{V_{d}}(m-m_{\text{out}}-m_{r}),\\
\dot{m} &={} \frac{A_{1}}{L_{c}}(\Pi(m,\omega) p_{s}-p_{d}),\\
\dot{\omega} &={} \frac{1}{J}(\tau-\tau_{c}),\\
\dot{m}_{r} &={} \frac{1}{\tau_{r}}(m_{\text{SP}}-m_{r}),
\end{aligned}
\label{eq:CompressorSystemCortinovisMultipleRecycle}
\end{equation}
where $p_{s}$ and $p_d$ are the suction and discharge suction pressures, $a_{01}$, $V_{s}$, $A_{1}$, $L_{c}$, $J$ are constant parameters defining the geometry of the compressor, the piping, and the shaft, $m_{\text{SP}}$ is the controller for the recycle valve, $m_{r}$ is the mass flow through the recycle valve, $m$ is the mass flow through the compressor, $\omega$ is the speed of the shaft of the compressor in rad\,s$^{-1}$, $\tau$ is torque provided by a flow controller, $\tau_{c}$ is the reaction torque of the compressor. The function $\Pi(\cdot,\cdot)$ gives the pressure ratio across a compressor as a function of compressor mass flow and speed:
\begin{equation}
\Pi(m,\omega) =   \alpha_0+\alpha_1m+\alpha_2\omega+\alpha_3m\omega+ \alpha_4{m}^2+\alpha_5\omega^2.
\label{eq:Map}
\end{equation}
The coefficients $\alpha_i$, $i=0,\ldots,5$ are usually estimated from operating data. Here we assume $\alpha_0=2.691$, $\alpha_1=-0.014$, $\alpha_2=-0.041$, $\alpha_3=0.0009$, $\alpha_4=0.0002$, $\alpha_5=0.00002$. The uncertainty in $\alpha_i$ is described in Section \ref{sec:CompressorUncertainties}.

The value of $m_{\text{in}}$ and $m_{\text{out}}$ captures the external mass flows on the suction and discharge side, respectively. The mass flows depend on the pressures $p_s$ and $p_d$, and external pressures $p_\mathrm{in}$ and $p_\mathrm{out}$:
\begin{subequations} \label{eqn:mass-flows}
\begin{align}
    {m}_{\mathrm{in}}&= 0.4k_\mathrm{in}A_\mathrm{in}\sqrt{p_\mathrm{in}-p_s},\\
    {m}_{\mathrm{out}} &= 0.8k_\mathrm{out}A_\mathrm{out}\sqrt{p_d-p_\mathrm{out}},\\
   {m}_{\mathrm{SP}} &= k_\mathrm{rec}u_{\text{rec}}A_\mathrm{rec}\sqrt{p_d-p_s},
\end{align}
\end{subequations}
where \(A_\mathrm{in}, A_\mathrm{out}, A_\mathrm{rec}\) represent the inlet, outlet and recycle valve orifice areas and \(k_\mathrm{in}, k_\mathrm{out}, k_\mathrm{rec}\) the respective valve gains.  The values of constant parameters were taken from~\cite{Real_Milosavljevic2020}. The value of $u_{\text{rec}}$ is obtained from an auxiliary PI controller and can take values between $0$ and $1$, ensured by a smoothed saturation function of the form \eqref{eq:SmoothSaturation} with $\beta_0=0.072$, $\beta_1=0.071$, $\beta_2=5.279$, $\beta_3=-0.001$.

\subsubsection{Optimal control}
The objective is to reach the desired flow level $m_d=100$ kg s$^{-1}$ without violating speed and flow constraints, imposed due to safety. The objective function was formulated as:
\begin{equation}
\label{eq:CompressorObjective}
J(\tau) = \int\limits_0^{t_f} 100m_{r}^2(s)+ 0.1\omega^2(s) +1000(m(s)-m_d)^2\d s
\end{equation}
where $t_f=100$ s.

The constraints on the mass flow and the speed are:
\begin{subequations} \label{eqn:CompressorConstraints}
\begin{align}
m\in&{} [65,105]\text{ kg s}^{-1}    \label{eqn:flow}\\
\omega\in&{}[550,876] \text{ rad s}^{-1} \label{eq:speed}
\end{align}
\end{subequations}

The control input $\tau$ is a PI controller parametrised by $K_p$ and $K_i$:
\begin{equation}
\tau(t) = K_p(m(t)-m_d)+K_i\int\limits_0^t (m(s)-m_d) \d s
\label{eq:CompressorControl}
\end{equation}
The parametrisation from \eqref{eq:CompressorControl} is typical for centrifugal compressors \cite{Real_Milosavljevic2020}. The torque that can be applied to the compressor must be between zero and 1000 Nm. The bounds on the torque were ensured by a smoothed saturation function of the form \eqref{eq:SmoothSaturation} with $\beta_0=73.324$, $\beta_1=0.072$, $\beta_2=0.005$, $\beta_3=0$.

\subsubsection{Uncertainties}
\label{sec:CompressorUncertainties}
The uncertainties we considered in this case study are in the valve gains $k_{\text{in}}$, $k_{\text{out}}$, $k_{\text{rec}}$, and correspond to $\pm5$\%, and in the parameters $\alpha_i$ in the polynomial compressor map \eqref{eq:Map}, and correspond to $\pm2$\%. Thus, there are nine uncertain parameters. 

\subsubsection{Results}
To find the flow controller from \eqref{eq:CompressorControl}, the dynamics were discretised using the trapezoidal collocation method with time step $0.5$\,s.

\begin{figure*}[tbp]
\centering
\includegraphics[width=0.9\textwidth]{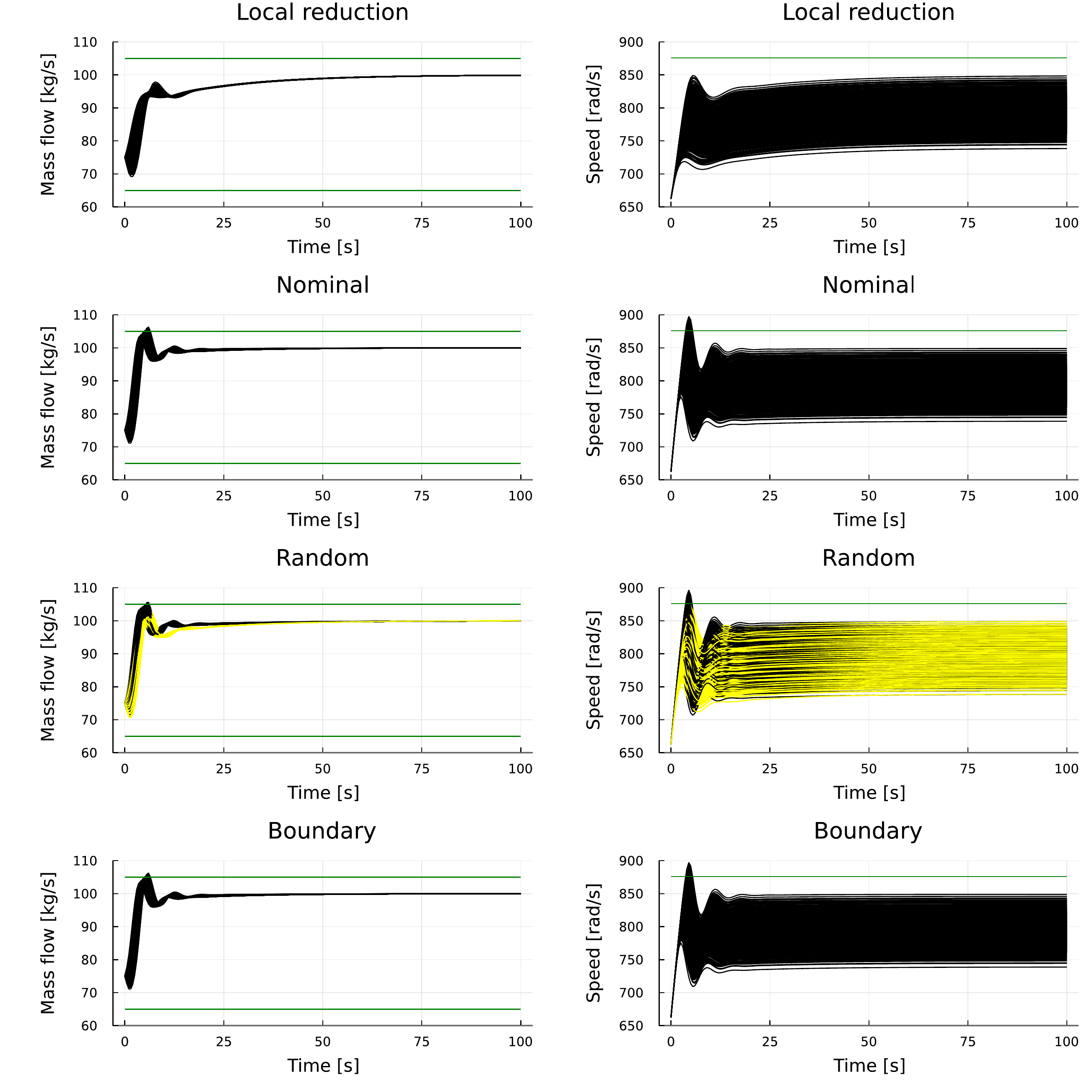}
\caption{Compressor case study - comparison of four approaches}
\label{fig:ComparisonCompressors}
\end{figure*}

\paragraph{Overall performance}
The local reduction method applied to the compressor case study resulted in two scenarios. Figure \ref{fig:ComparisonCompressors} shows the results of the validation for the two scenarios obtained from the local reduction method with $\epsilon=10^{-6}$. The controller obtained from the local reduction did not violate the constraints on either the mass flow (top left) or the speed of the compressor (top right).

\paragraph{Comparison with other approaches}
If we were to consider all extreme realisations of the uncertainties, we would obtain $2^9=512$ scenarios. The local reduction method we propose in this paper reduced the number of scenarios to two. Conversely, the nominal controller was not able to satisfy the constraints and both the mass flow and the speed violated their upper limits (second row of plots in Fig. \ref{fig:ComparisonCompressors}). The controllers based on randomly chosen scenarios (from a uniform distribution) are shown in the third row, with black indicating a controller based on two random scenarios. The controller based on two random scenarios was insufficient to ensure constraint satisfaction and both the mass flow and the speed of the compressor violated their upper limits. The minimal number of random scenarios needed to ensure constraint satisfaction was nine (yellow). Finally, the bottom row in Fig. \ref{fig:ComparisonCompressors} shows the performance of the nominal+extreme controller obtained for three scenarios (one scenario on the lower bound, one scenario on the upper bound, and one scenario with no uncertainty). The controller based on the nominal and two extreme scenarios was insufficient to satisfy constraints. Thus, the comparison with other scenario-based approaches confirms the potential of local reduction for solving robust optimal control problems with parametric uncertainty.

\section{Conclusions}
\label{sec:Conclusions}

Solving robust nonlinear optimal control problems is challenging, especially if the knowledge about the uncertainty is limited. Scenario-based approaches provide a way of reformulating the optimal control problems as nonlinear optimization problems. The choice of scenarios and their number affects the robustness of the solution as well as computational complexity of the resulting optimisation problems. In this work, we formulated robust optimal control problems with time-varying and parametric uncertainty as semi-infinite optimisation problems to facilitate the choice of scenarios. The new formulation enabled usage of semi-infinite optimisation algorithms, such as local reduction methods. By adding interim worst-case scenarios, the local reduction method enables finding a trade-off between the size of the resulting optimization problem and robustness of the solution to the original optimal control problem. We overcome the dependence on global solvers in the original local reduction formulation by proposing inexact local reduction and providing theoretical bounds on possible constraint violation. The new method consists in solving multiple optimal control problems of reduced size compared to the full scenario-based optimisation. In particular, the small control problems can be solved in parallel, further improving the computational speed.

The performance of our approach was evaluated in two case studies with both additive and parametric uncertainty: thermal comfort control in a residential building and mass flow control in a centrifugal compressor. A comparison with common approaches based on a random choice of scenarios and on extreme scenarios indicates that local reduction allows solving robust optimal control problems in an efficient way while ensuring robustness. In particular, the case studies confirm that the proposed inexact local reduction method allows finding worst-case scenarios in the interior of the uncertainty sets. As a result, the new method was able to handle larger parametric uncertainty than other scenario-based approaches.

In this work we required that the constraints must be satisfied for all realisations of the uncertainty. In the future, it would be advisable to look at the conservatism of the obtained solutions and possible relaxations of this requirement. In particular, tighter bounds on constraint violation can be derived if a distribution of the uncertainty is available. Future work could include numerical improvements of approximate local reduction, including warm-starting and use of custom nonlinear optimization solvers, as well as explicit parallelisation of the optimal control problems.


\bibliography{bibTAC}%
\end{document}